\documentclass[reqno]{amsart}
\usepackage{amsmath, amsfonts, amssymb, amsthm, amscd}
\usepackage{hyperref}

\newtheorem{theorem}{Theorem}[section]
\newtheorem{lemma}[theorem]{Lemma}
\newtheorem{corollary}[theorem]{Corollary} 

\newtheorem{proposition}[theorem]{Proposition}

\theoremstyle{definition}
\newtheorem{definition}[theorem]{{Definition}}
\newtheorem{example}[theorem]{Example}
\newtheorem{remark}[theorem]{Remark}

\newtheorem{fact}[theorem]{\bf Fact}

\newtheorem*{chunk*}{}

\numberwithin{equation}{theorem}

\usepackage[all]{xy}
\usepackage{enumerate}

\usepackage{tikz}
\usetikzlibrary{shapes}
\usetikzlibrary{snakes}

\usepackage{graphicx}

\usepackage{pgf}
\usepackage{latexsym,amsmath,amssymb}
\usepackage[parfill]{parskip} 

\newif\ifdviwin

\usepackage[latin1]{inputenc}
\usepackage[english]{babel}
\usepackage{indentfirst}
\usepackage[mathscr]{eucal}
\usepackage{amsmath,amscd}
\usepackage{xxcolor}
\usepackage{color}

\def\B{{\mathcal B}}

\def\m{{\mathfrak m}}

\def\Tor{\operatorname{Tor}}

\def\im{\operatorname{im}}

\def\Cl{\operatorname{Cl}}

\def\dim{\operatorname{dim}}

\def\cone{\operatorname{Cone}}
\def\length{\operatorname{length}}

\def\im{\operatorname{im}}

\newcommand{\excise}[1]{}

\newcommand{\ba}{{\bf a}}
\newcommand{\bb}{{\bf b}}
\newcommand{\bt}{{\bf t}}
\newcommand{\bu}{{\bf u}}
\newcommand{\pmba}{{\pmb{\alpha}}}
\newcommand{\pmbb}{{\pmb{\beta}}}

\newcommand{\fq}{\mathfrak q}

\newcommand{\fm}{\mathfrak m}

\newcommand{\sfk}{\mathsf k}

\newcommand{\cB}{\mathcal B}

\newcommand{\vol}{\operatorname{Vol}}
\newcommand{\rvol}{\operatorname{rel-vol}}
\newcommand{\gp}{\operatorname{gp}}

\newcommand{\bolde}{\bf e}

\newcommand{\hk}{{\rm e}_{HK}}
\newcommand{\e}{{\rm e}}

\newcommand{\showtwo}[3]
{%
    \begin{center}
        \includegraphics[width=#3\linewidth]{#1}%
        \hskip2in
        \includegraphics[width=#3\linewidth]{#2}%
    \end{center}
}

\begin{document}
\bibliographystyle{amsplain}

\subjclass[2010]{Primary 13A15, 05E40, 20M25; Secondary 52B11 }
\title[Invariants of Intersection Algebras of Principal Monomial Ideals]
{Computing The Invariants of Intersection Algebras of Principal Monomial Ideals}

\thanks{
The second author was supported in part by a grant (\#245926) from the Simons Foundation.}

\author[F.\ Enescu]{Florian Enescu}
\address{Florian Enescu, Department of Mathematics and Statistics,
758 COE,
Georgia State University,
Atlanta, Georgia 30303, USA}
\email{fenescu@gsu.edu}
\urladdr{http://www2.gsu.edu/$\sim$matfxe/}

\author[S.\ Spiroff]{Sandra Spiroff}
\address{Sandra Spiroff, Department of Mathematics, 
335 Hume Hall,
University of Mississippi, 
University, MS  38677-1848, USA}
\email{spiroff@olemiss.edu}
\urladdr{http://home.olemiss.edu/$\sim$spiroff/}

\begin{abstract}  We continue the study of intersection algebras $\mathcal B = \mathcal B_R(I, J)$ of two ideals $I, J$ in a commutative Noetherian ring $R$.  In particular, we exploit the semigroup ring and toric structures in order to calculate various invariants of the intersection algebra when $R$ is a polynomial ring over a field and $I,J$ are principal monomial ideals.  Specifically, we calculate the $F$-signature, divisor class group, and Hilbert-Samuel and Hilbert-Kunz multiplicities, sometimes restricting to certain cases in order to obtain explicit formul{\ae}. This provides a new class of rings where formul{\ae} for the $F$-signature and Hilbert-Kunz multiplicity, dependent on families of parameters, are provided.
\end{abstract}

\date \today
\maketitle

\section*{Introduction}

The intersection algebra $\mathcal B = \mathcal B_R(I, J)$ of two ideals $I, J$ in a commutative Noetherian ring $R$ was first introduced by J.~B.~Fields \cite{Fi} in 2002 in his study of the function $(r, s) \to \length(\Tor_1^R(R/I^r, R/J^s))$.  Fields has shown that this algebra is finitely generated when $I,J$ are monomial ideals in a polynomial ring in finitely many variables over a field. Not much is known about the properties of $\mathcal B_R(I,J)$ when $I,J$ are arbitrary ideals in a Noetherian ring $R$. The intersection algebra, which was the topic of S.~Malec's dissertation \cite{M-thesis}, was shown to be finitely generated over $R$ in the case that the two ideals are principal ideals in a unique factorization domain and an algorithm was given that produces generators for this $R$-algebra.  Of particular interest is the case where $R$ is a polynomial ring over a field $\sfk$ and $I$, $J$ are principal monomial ideals. In this case, the properties of the intersection algebra are a reflection of the combinatorics of the monomial exponents, which is often complicated. This project is a continuation of the study undertaken by Malec and Enescu \cite{EM}.   In particular, we aim to calculate the $F$-signature, the Hilbert-Samuel and Hilbert-Kunz multiplicities, and the divisor class group of $\cB$.  

The scope of the paper is to present general formul{\ae} for the invariants described above as functions of the monomial exponents. The $F$-signature and Hilbert-Kunz multiplicity are difficult to compute even for simple classes of rings of small dimension. In fact, the literature does not contain many examples when both of these invariants can computed simultaneously. Therefore, having general formul{\ae} for them for intersection algebras is extremely valuable. Our work provides such formula{\ae} as functions of parameters, giving researchers the opportunity to test their questions about these invariants, while varying the parameters. Since intersection algebras are toric, approaching the topic from the point of view of toric ideals and dual cones, we take advantage of the wealth of material in the literature, namely \cite{BG} and \cite{CLS}, and use software, such as {\it Mathematica} \cite{Wolf} and MATLAB \cite{MAT}, when possible.   We were able to accomplish this computation in a few important cases, while showing the limitations of current literature on toric rings, which might come as unexpected to some algebraists. 

Let $R = \sfk[x_1, \dots, x_n]$ over a field $\sfk$ and $I = (x_1^{a_1}x_2^{a_2}\cdots x_n^{a_n})$ and $J = (x_1^{b_1}x_2^{b_2}\cdots x_n^{b_n})$, where the exponents are nonnegative integers.
For general $n$, we compute the Hilbert-Samuel multiplicity and divisor class group for $\mathcal B_R(I, J)$. A closed formula in $n$, and $\ba = (a_1, \dots, a_n), \bb = (b_1, \dots, b_n)$, in the general case for the $F$-signature and Hilbert-Kunz multiplicity could not be obtained,
even with the high-performance computing capabilities offered by Georgia State University and, independently, the University of Mississippi. Whether this is possible remains for now an open problem. One challenge is the lack of a useful description of the Hilbert basis elements $\mathcal H$ in terms of $\ba, \bb$. The other, more serious one, is that, even when we know a description of the elements in $\mathcal H$, the software has difficulty either completing an integration that depends upon multiple parameters, or providing meaningful results. In particular, even when $n$ is simply equal to two, a myriad of cases arise in the geometry of the $F$-signature, preventing one from obtaining a general result. While we are able to provide some general formul{\ae}, because of the difficulty and technicality of the computations involved, we obtain the most complete results when each $a_i$ is a fixed multiple of $b_i$; i.e., $a_i = kb_i$ $\forall i$ and $k \in \mathbb N_+$, or when $n =1$.  If $k = 1$, then each of the algebras is $\mathbb Q$-Gorenstein, and in fact, a hypersurface, and we are able to provide formul{\ae} for the $F$-signature and Hilbert-Kunz multiplicity. When $n=1$, we compute the $F$-signature for a pair of positive integers $a, b$, with $I = (x^a), J = (x^b)$, and provide a formula for the Hilbert-Kunz multiplicity in some special cases.  In addition, we detail how one can use computer software to calculate the multiplicity for specific numerical examples.

The paper is laid out as follows.  The first section gives the necessary background and preliminary results in preparation of exploiting the relevant literature.  The second section details the calculation the Hilbert-Samuel multiplicity.  In the third section, we provide formul{\ae} for the $F$-signature; in the fourth we compute the Hilbert-Kunz multiplicity, and in the Appendix we detail how {\it Mathematica} may be used in specific numerical computations and we provide 3-D renderings of the volumes in question.

\section*{Acknowledgements}
The authors thank Sara Malec for useful conversations in the early stages of this project. In terms of assistance with software and running code on high-performance computers, when needed, the authors thank the contacts at Wolfram Technical Support, Semir Sarajlic and Suranga Edirisinghe Pathirannehelage at Georgia State, and Brian Hopkins and Ben Pharr at the University of Mississippi.

%%%%%%%%%%%%%%%%%%%%%%%%%%%%%%%%%%%%%%%%%%%%%%%%%

\section{Preliminaries, Notation, and Basic Results}

%%%%%%%%%%%%%%%%%%%%%%%%%%%%%%%%%%%%%%%%%%%%%%%%%

{\it The notation set in this section, especially that of Definition \ref{Hilbertsets}, Proposition \ref{positive}, and Theorem \ref{cones}, will be used throughout the paper}.  We provide a brief review of the relevant definitions and results.  More background, details, and discussion can be found in \cite{M} and \cite{EM}. Throughout the paper, $\mathbb N = \mathbb{Z}_{\geq 0}$, $\mathbb N_+ = \mathbb{Z}_{> 0}$, $\sfk$ is a field, and $x_1, \ldots, x_n$ will denote indeterminates over $\sfk$.

\begin{definition} \label{ring-assump} 
Let $R$ be a commutative ring and $I, J$ be two ideals of $R$. The intersection algebra of $I$ and $J$ is $$\mathcal B_R(I,J) = \bigoplus_{r, s \in \mathbb N}(I^r \cap J^s).$$ 

If $u, v$ are taken to be two indexing variables, then $\mathcal B_R(I,J) = \sum_{r, s \in \mathbb N}(I^r \cap J^s)u^rv^s \subseteq R[u,v]$.  

In particular, set $R = \sfk[x_1, \dots, x_n]$ and $I = (x_1^{a_1}x_2^{a_2}\cdots x_n^{a_n})$ and $J = (x_1^{b_1}x_2^{b_2}\cdots x_n^{b_n})$. Without loss of generality, when no pair $a_i, b_i$ is simultaneously zero for any $i =1, \ldots, n$, we may assume that the strings of nonnegative integers $\ba= a_1, \dots, a_n$ and $\bb = b_1, \dots, b_n$ are {\bf fan ordered}, as per \cite[Definition 2.7]{M}; i.e., $\displaystyle{\frac{a_i}{b_i} \geq \frac{a_{i+1}}{b_{i+1}}}$ for all $i = 1, \dots, n$, where by convention we write $\frac{k}{0} = \infty$ for $k >0$. We will denote the associated intersection algebra by $\B(\ba, \bb)$, or more briefly, $\mathcal B$. One should note that $\B(\ba, \bb)$ is $\mathbb{N}$-graded with $\m=(x^t : t\in Q)$ as the unique homogeneous maximal ideal. Let $a_0=1, b_0=0, a_{n+1}=0$, and $b_{n+1}=1$. If $\displaystyle{\frac{a_i}{b_i} > \frac{a_{i+1}}{b_{i+1}}}$ for all $i = 0, \dots, n$, then we call the fan {\bf non-degenerate}. This implies, in particular, that $b_i$ is nonzero for all $i \geq 1$.
\end{definition}

For $i=0, \dots, n$, let $C_i =  \{ \lambda_1(b_i, a_i) + \lambda_2(b_{i+1},a_{i+1}) | \lambda_i \in \mathbb R_{\geq 0} \}$. We refer to $\sum_{\ba, \bb}$, formed by these cones and their faces, as the fan of $\ba$ and $\bb$.  (See \cite[pp.~3-4]{M} for more details.)  Each segment $Q_i = C_i \cap \mathbb{N}^2$ of the fan admits a unique Hilbert basis $\mathcal H_i = \{ (r_{i,j}, s_{i, j}): j= 1, 2, \dots, n_i \}$.  Denote its cardinality by $h_i$.  

\begin{definition} \label{Hilbertsets} The set $\mathcal H= \cup _{i=0, \dots, n} \mathcal H_i$ is called the {\it Hilbert set} for $\B(\ba, \bb)$ and its cardinality, denoted by $h$, will be called the {\it Hilbert number} for $\B(\ba, \bb)$. Note that $h = \sum_i h_i -n$ when the fan is non-degenerate. For every $v=(r,s) \in \mathbb{Z}^2$, let $\bt(v) = (\max(a_ir, b_is))_{i=1, \ldots, n}$.  Set $\mathcal G= \{ (v, \bt(v)) : v \in \mathcal H\}$.
\end{definition}

Let $Q=Q(\ba, \bb) = \{ (r, s, t_1, \ldots, t_n) : t_i  \geq \max(a_i r, b_i s), i =1, \dots, n\} \subseteq \mathbb{N}^{n+2}$.  Then $Q$ is a commutative semigroup with identity (or, in other words, a commutative monoid).  As per \cite[p.~50]{BG}, let $\gp(Q)$ be the smallest (abelian) group that contains $Q$.

\begin{lemma}
\label{group}
With the notations above, $\gp(Q(\ba, \bb)) = \mathbb{Z}^{n+2}$.

\end{lemma}
\begin{proof}
Note that that if $(r, s, t_1, \ldots, t_i, \ldots,  t_n) \in Q=Q(\ba, \bb)$, then $(r, s, t_1, \ldots, t_i+1, \ldots,  t_n) \in Q$ and so the standard basis vector $\bold{e}_{i+2}= (r, s, t_1, \ldots, t_i+1, \ldots,  t_n)-(r, s, t_1, \ldots, t_i, \ldots,  t_n)\in \gp(Q)$, for $i =1, \dots, n$. A similar argument shows that $\bold{e}_1, \bold{e}_2 \in Q$ as well.
\end{proof}

The following has been noted in \cite[Theorem 3.5]{M} and \cite[Theorem 1.11]{EM}, and is key to our purpose.

\begin{fact} \label{old}
$\B (\ba, \bb) = \sfk [Q]$, and it is generated as a $\sfk$-algebra by $\{x_1, \ldots, x_n\}$ and the monomials with exponent vectors from $\mathcal G$. Moreover, the minimal number of generators of $\fm$, the ideal in $\mathcal B$ generated by the monomials in $Q$, equals $n + h$.  This is also called  the embedding dimension of $\cB$, denoted $\nu(\mathcal B)$.
\end{fact}

\begin{theorem} \label{properties}
The intersection algebra $\B=\sfk[Q(\ba, \bb)]$ is normal, Cohen-Macaulay, $F$-regular, and of dimension $n+2$.
\end{theorem}

\begin{proof} The fact that $\B$ is normal and Cohen-Macaulay is \cite[Proposition 1]{Ho} and \cite[Theorem 1]{Ho}, respectively.  The fact that $\dim \B = n+2$ is \cite[Proposition 1.3]{EM}. The $F$-regularity property follows from~\cite{Ho} too, since normal toric rings are direct summands in a regular ring, hence they are $F$-regular.
\end{proof}

It is useful to note that, in general, the study of intersection algebras $\cB (\ba, \bb)$ can be reduced to the case of vectors $\ba, \bb$ with positive entries.

\begin{proposition} \label{positive}
Let $\ba= (a_1, \dots, a_n)$ and $\bb = (b_1, \dots, b_n)$ be nonnegative integer vectors. Fix $0 \leq i,  j \leq n$ such that $b_1=\cdots=b_i=0$, while $b_l \neq 0$ for all $l > i$, and $a_k \neq 0$ for all $k < j$, while $a_j =\cdots =a_n =0$.

\begin{enumerate}
\item 
If $i \geq j$, then $$\cB (\ba, \bb) \simeq \sfk[x_1, \ldots, x_n, U, V],$$ where $U,V$ are indeterminates over $\sfk$.

\item

If $i < j$, then

$$\cB (\ba, \bb) \simeq \cB((a_{i+1}, \cdots, a_{j-1}), (b_{i+1}, \cdots, b_{j-1})) [ x_1, \ldots, x_i, x_j, \ldots, x_n],$$ with the convention that if $b_1 \neq 0$ (or $a_n \neq 0$) we do not adjoin the variables $x_{i'}, i' \leq i $ (respectively,  $x_{j'}, j' \geq j$).
\end{enumerate}
\end{proposition}
\begin{proof}
We will prove (2), since (1) can be shown in a similar fashion. So, let us assume that $i < j$.

Note that $$x_1^{a_1r}\cdots x_i^{a_ir} \cdot x_{i+1}^{\max(a_{i+1}r, b_{i+1}s)} \cdots x_{j-1}^{\max(a_{j-1}r, b_{j-1}s)} x_j^{b_js} \cdots x_n^{b_ns} \cdot u^r v^s$$ can be rewritten as
$$x_{i+1}^{\max(a_{i+1}r, b_{i+1}s)} \cdots x_{j-1}^{\max(a_{j-1}r, b_{j-1}s)}  \cdot (x_1^{a_1}\cdots x_i^{a_i} \cdot u)^r (x_j^{b_j} \cdots x_n^{b_n}v)^s.$$
Consider the intersection algebra over $\sfk$ corresponding to the two vectors $(a_{i+1}, \cdots, a_{j-1}), (b_{i+1}, \cdots, b_{j-1})$, where we use $U, V$ as indeterminates to distinguish them from $u,v$.  Hence $$\cB((a_{i+1}, \cdots, a_{j-1}), (b_{i+1}, \cdots, b_{j-1}))= \sum_{r,s}  (x_{i+1}^{\max(a_{i+1}r, b_{i+1}s)} \cdots x_{j-1}^{\max(a_{j-1}r, b_{j-1}s)}) U^rV^s \subseteq \sfk[x_{i+1}, \ldots, x_{j-1}, U, V].$$ 

Map the two new indeterminates $U,V$ to $x_1^{a_1}\cdots x_i^{a_i} \cdot u$ and $x_j^{b_j} \cdots x_n^{b_n}v$, respectively, and $x_l \to x_l$, for $l=1, \ldots, n$, which results in a homomorphism:
$$\sfk [x_1, \cdots, x_n, U, V] \to \sfk [x_1, \cdots, x_n, u, v].$$
By restriction, this induces a surjective homomorphism:
$$\cB((a_{i+1}, \cdots, a_{j-1}), (b_{i+1}, \cdots, b_{j-1})) [ x_1, \ldots, x_i, x_j, \ldots, x_n] \to \cB (\ba, \bb)$$ of rings of the same dimension, which must therefore be an isomorphism.
\end{proof}

\medskip
\noindent
\begin{remark} In light of Proposition~\ref{positive}, for the remainder of this paper, we will assume that we are given two integer vectors $\ba, \bb \in \mathbb{Z}^n$ with all entries positive, which are fan ordered.
\end{remark}

\subsection*{The dual cone and primitive vectors}
Because the aim is to apply results from \cite{CLS} and \cite{W} (among others), which take a toric approach, it is necessary to present the perspective of the dual cone.  This requires some additional notation, which will be referenced in subsequent sections.

\medskip
\noindent
{\bf Notation.} \label{cone-assump} Given the strings of positive integers $a_1, \dots, a_n$ and $b_1, \dots, b_n$, which are are fan ordered, set the following notation in $\mathbb Z^{n+2}$:
$${\bolde}_1 = (1,0,0,0,\dots,0), \hskip.08in {\bolde}_2 = (0,1,0,0,\dots,0),$$ 
$$ \pmba_1 = (-a_1, 0, 1,0,\dots, 0), \hskip.08in \pmba_2 = (-a_2, 0, 0, 1,0,\dots,0),\dots, \hskip.08in \pmba_n = (-a_n,0,0,\dots,0,1),$$ 
$$\pmbb_1 = (0, -b_1,1,0,\dots,0), \hskip.08in \pmbb_2 = (0,-b_2,0,1,0,\dots,0),\dots, \hskip.08in \pmbb_n = (0,-b_n,0,\dots,0,1).$$

The cone in $\mathbb R^{n+2}$ generated by $\bolde_1, \bolde_2, \pmba_1, \dots, \pmba_n, \pmbb_1, \dots, \pmbb_n$ will be denoted by  $\sigma$.

Let $\lambda_i =\frac{a_i}{b_i}, \lambda_i \geq \lambda_{i+1},  i=1, \dots, n.$  Additionally, take $\lambda_{n+1} =0$ (and $\lambda_0 =\infty$).  Set:
$${\bolde}_3 = (0,0,1,0, \dots, 0), \hskip.25in \dots, \hskip.25in {\bolde}_{n+2} = (0, \dots, 0,1),$$ 
$${\bf w}_0 = (1, 0, a_1, \dots, a_n), \hskip.4in {\bf w}_{n+1} = (0, 1, b_1, \dots, b_n)$$
$${\bf w}_1 = (1, \lambda_1, a_1, \lambda_1b_2, \lambda_1 b_3, \dots, \lambda_1 b_n),$$
$${\bf w}_2 = (1, \lambda_2, a_1, a_2, \lambda_2b_3, \dots, \lambda_2 b_n),$$ 
\centerline{$\vdots$}
$${\bf w}_n = (1, \lambda_n, a_1, a_2, a_3, \dots, a_{n-1}, a_n).$$

Let $C = \{ (r,s, t_1, \ldots, t_n): r, s \in \mathbb R_{\geq 0}, \hskip.08in t_i \geq \max(a_ir, b_is), i=1, \ldots, n\}.$ Clearly, $Q= C \cap \mathbb{Z}^{n+2}.$

\begin{theorem} \label{cones}
With the notation as above, $$C= \cone({\bolde}_3, \dots, {\bolde}_{n+2}, {\bf w_0}, {\bf w_1}, \dots, {\bf w}_n, {\bf w}_{n+1}),$$ hence $C$ is a convex polyhedral cone. Moreover, the cones $$\sigma=\cone({\bolde}_1, {\bolde}_2, \pmba_1, \dots, \pmba_n, \pmbb_1, \dots, \pmbb_n) \quad \text{and} \quad C$$ are dual to each other; i.e., $\sigma^{\vee} =C$.
\end{theorem}
\begin{proof} Obviously, $\cone({\bolde}_3, \dots, {\bolde}_{n+2}, {\bf w_0}, {\bf w_1}, \dots, {\bf w}_n, {\bf w}_{n+1}) \subseteq C$, so it remains to prove the reverse inclusion.  Let $(r,s, t_1, \ldots, t_n) \in C$. If $r=0$, the claim is easy to show. Say $r \neq 0$ and let $i$, $0 \leq i \leq n+1$, such that $\lambda_i \geq \frac{s}{r} > \lambda_{i+1}$. This implies that
$$\max(a_kr, b_ks) = \left\{
\begin{array}{ll}
a_kr, & \mbox{if} \ k \leq i, \\
\ b_ks, & \mbox{if} \ k \ \geq i+1.
\end{array}
\right.
$$

Since $t_k \geq \max(a_kr, b_ks)$, for all $k=1, \dots, n$, it is enough to show that
$$u:=(r,s, \ldots, \max(a_kr, b_ks), \ldots) \in C.$$
Let $\gamma, \theta$ be nonnegative numbers such that 
$$ r = \gamma +\theta, \qquad s = \lambda_i \gamma + \lambda_{i+1} \theta.$$
We claim that $$ u =\gamma {\bf w}_i + \theta {\bf w}_{i+1},$$ which shows that $u$ is an element in $\cone({\bolde}_3, \dots, {\bolde}_{n+2}, {\bf w_0}, {\bf w_1}, \dots, {\bf w}_n, {\bf w}_{n+1})$. Note that $\gamma {\bf w}_i + \theta {\bf w}_{i+1}$ has, on position $k+2$, the entry
$$\gamma a_k + \theta a_k =a_kr, {\rm if} \ k \leq i,$$ and
$$\gamma \lambda_i b_k + \theta \lambda_{i+1} b_k =b_ks, {\rm if} \  k \geq i+1.$$ In both cases, we get $\max(a_k r, b_k s)$ which matches the entry on position $k+2$ for $u$. 

Regarding the second part of the statement, for $m \in \mathbb{R}^{n+2}$, let $H^+_m =\{ u \in \mathbb{R}^{n+2} \mid \langle m, u \rangle \geq 0\}$. It can be verified that $C = H^+_{\bolde_1} \cap H^+_{\bolde _2} \cap H^+_{\pmba_1} \cap  \cdots \cap H^+_{\pmba_n} \cap H^+_{\pmbb_1} \cap \cdots \cap H^+_{\pmbb_n}.$  By definition, ${\bolde}_1, {\bolde}_2, \pmba_1, \dots, \pmba_n, \pmbb_1, \dots, \pmbb_n$ are in $C^{\vee}= \{ m \in \mathbb R^{n+2} \mid \langle m, u \rangle\ \geq 0, \ {\rm for \ all} \ u \in C\}$, and Proposition 1.2.8 (a) in~\cite{CLS} applies to give that $C^{\vee}=\cone({\bolde}_1, {\bolde}_2, \pmba_1, \dots, \pmba_n, \pmbb_1, \dots, \pmbb_n).$

\end{proof}

\begin{corollary}
\label{primvect} The intersection algebra $\mathcal B$ equals $\sfk[\sigma^{\vee} \cap \mathbb{Z}^{n+2}]$. 
\end{corollary}

In summary, we note that the intersection algebra $\mathcal B = \cB (\ba, \bb) = \sfk[Q] =\sfk[\sigma^{\vee} \cap \mathbb{Z}^{n+2}]$, and how we choose to describe $\mathcal B$ depends upon our perspective.

\medskip

\begin{remark} \label{vectors} The vectors ${\bolde}_1, {\bolde}_2$, $\pmba_1, \pmba_2,\dots, \pmba_n, \pmbb_1,\dots, \pmbb_n$ are the 
{\it primitive vectors} of $\mathcal B$ in the sense \cite[p.~2708]{HWY} that they are a minimal system of generators for $\sigma$, none can be replaced with a multiple of $r \in \mathbb R$ with $0 < r < 1$, and $\sigma^{\vee} = \{u \in \mathbb R^{n+2} \mid \langle u, v \rangle \geq 0\}$.  The vectors ${\bolde}_3, \dots, {\bolde}_{n+2}, {\bf w_0}, {\bf w_1}, \dots, {\bf w}_n, {\bf w}_{n+1}$ determine the {\it nonnegative rays of the cone} $C$ and are the generators of this cone. (See Fact \ref{HScalc}.)
\end{remark}

%~~~~~~~~~~~~~~~~~~~~~~~~~~~~~~~~~~~~~~~~~~~~~~~~~~~~~~~~~~~~~~~~~~
\subsection{$\mathbb Q$-Gorenstein, Gorenstein, and the Divisor Class Group}
%~~~~~~~~~~~~~~~~~~~~~~~~~~~~~~~~~~~~~~~~~~~~~~~~~~~~~~~~~~~~~~~~~~

\begin{proposition} \label{classgroup} Let $Q=Q(\ba, \bb)$ be the monoid in $\mathbb Z^{n+2}$, as above.  Then $$\Cl(\cB(\ba, \bb)) = \Cl(Q) \cong \mathbb Z^n.$$
\end{proposition}

\begin{proof} We apply \cite[Corollary 4.56]{BG}.  (See also \cite{Ch}.)  Since there are $2n+2$ facets of $C = H^+_{\bolde_1} \cap H^+_{\bolde _2} \cap H^+_{\pmba_1} \cap  \cdots \cap H^+_{\pmba_n} \cap H^+_{\pmbb_1} \cap \cdots \cap H^+_{\pmbb_n}$, the standard map $Q \to \mathbb N_+^{2n+2}$ (see \cite[p.~55]{BG}) is defined by the matrix $A = [{\bolde}_1^T \hskip.05in {\bolde}_2^T \hskip.05in \pmba_1^T \cdots \pmba_n^T \hskip.05in \pmbb_1^T \cdots \pmbb_n^T]$.  In particular, multiplication by $A$  gives a $\mathbb Z$-linear map from $\gp(Q(\ba, \bb))$, which equals $\mathbb{Z}^{n+2}$ by Lemma \ref{group}, to $\mathbb{Z}^{2n+2}$.  Since $A$ is row equivalent to $U = [{\bolde}_1^T \cdots {\bolde}_{n+2}^T \hskip.05in {\bf 0} \cdots {\bf 0}]$, it follows that $\Cl(Q) \cong \mathbb Z^{2n+2}/\im(U) \cong \mathbb Z^n$.  Finally, since $\mathcal B = \sfk[Q]$, by \cite[Corollary 4.60]{BG}, $\Cl(\mathcal B) = \Cl(\sfk) \oplus \Cl(Q) \cong\Cl(Q)$. 
\end{proof}

\begin{proposition}
\label{q-gorenstein}
 The intersection algebra $\mathcal B$ is $\mathbb Q$-Gorenstein if and only if $a_i=b_i$ for each $i = 1, \dots, n$. Moreover, in this case, the intersection algebra $\cB$ is a hypersurface and $$\mathcal B = \sfk [x_1, \ldots, x_n, A, B, C]/(AB-x_1^{a_1} \cdots x_n^{a_n} C).$$
\end{proposition}

\begin{proof}  Recall that $\mathcal B = \mathsf k[\sigma^{\vee} \cap \mathbb Z^{n+2}]$.  By \cite[Lemma-Definition 4.4]{HWY}, the aim is to find a vector $\omega \in Q \otimes_{\mathbb Z} \mathbb R$ such that $\langle \omega, - \rangle = 1$ for all of the primitive generators of $\sigma$; i.e., for ${\bolde}_1, {\bolde}_2$, and for all $\pmba_i, \pmbb_j$ above.  Set $\omega = (y_1, y_2, z_1, z_2, z_3, \dots, z_n)$.  From $\langle \omega, {\bolde}_1 \rangle = \omega \cdot (1,0,0,\dots,0) = 1$ it follows that $y_1 = 1$; likewise, $y_2=1.$  Next, from $\langle (1, 1, z_1,\dots, z_n), \pmba_i \rangle = 1$, one obtains $-a_i+z_i = 1$; i.e., $z_i = 1+a_i$.  Similarly, $z_i = 1+b_i$.  Thus, obviously $a_i = b_i$.

The set $\mathcal G$ (see Definition \ref{Hilbertsets}) for $\mathcal B$ is easy to write down in this case: $\mathcal G = \{ (1, 0, a_1, \ldots, a_n), $ $(0,1, a_1, \ldots, a_n), (1, 1, a_1, \ldots a_n) \},$ and hence $\mathcal B= \sfk [x_1, \ldots x_n, x_1^{a_1} \cdots x_n^{a_n} u, x_1^{a_1} \cdots x_n^{a_n} v, x_1^{a_1} \cdots x_n^{a_n} uv]$, which has dimension $n+2$. Map $\sfk [x_1, \ldots, x_n, A, B, C]$ to $\mathcal B$ by sending $x_i$ to $x_i$, for $i=1, \ldots, n$; $A$ to $x_1^{a_1} \cdots x_n^{a_n} u$; $B$ to $x_1^{a_1} \cdots x_n^{a_n} v$; and $C$ to $x_1^{a_1} \cdots x_n^{a_n} uv$. The kernel is a one dimensional prime ideal, hence principal. It can easily be seen that its generator is $AB-x_1^{a_1} \cdots x_n^{a_n} C$.
\end{proof}

\medskip
\begin{remark}
This result generalizes and extends Corollary 2.23 in~\cite{EM}.
\end{remark}

%%%%%%%%%%%%%%%%%%%%%%%%%%%%%%%%%%%%%%%%%%%%%%%%%

\section{Computing the Hilbert-Samuel Multiplicity}

%%%%%%%%%%%%%%%%%%%%%%%%%%%%%%%%%%%%%%%%%%%%%%%%%

The main result of this section is Theorem \ref{generalHS}, which computes the Hilbert-Samuel multiplicity of $\mathcal B = \sfk[Q]$ by calculating the volume of a union of polytopes.  The first two items provide the necessary machinery for our proof.

\begin{fact} \cite[p.~3]{BEF} \label{Vformula} For vertices $v_0, \dots, v_d \in \mathbb R^d$, let $\Delta(v_0, \dots, v_d)$ denote the associated simplex in $\mathbb R^d$.  Then
$$\vol(\Delta(v_0, \dots, v_d)) = \frac{\det(v_1-v_0, v_2-v_0, \dots, v_d - v_0)}{d!}.$$
\end{fact}

\begin{fact} \label{HScalc} \cite[Theorem 6.55]{BG} Suppose $Q \subseteq \mathbb Z^d$ is a positive affine monoid of rank $d$, and $I$ is a monomial ideal in $R = \sfk[Q]$ primary to the maximal ideal generated by all monomials $\neq 1$.   Let $\mathcal P(I)$ be the convex hull of the lattice points corresponding to the monomials in $I$ and $\mathcal C$ the cone generated by the nonnegative rays determined by the generators of $Q$. 
Then,
$$[\mathbb Z^d: \gp(Q)] \cdot \e(I, R) = d! \vol(\mathcal C \setminus \mathcal P(I)) \quad {\text{and}} \quad \vol(\mathcal C \setminus \mathcal P(I)) = \lim_{n \to \infty} \frac{\#(\mathcal C \backslash n \mathcal P(I) \cap \mathbb Z^d)}{n^d},$$
where although $\mathcal P(I)$ is an unbounded polyhedron in $\mathbb Z^{n+2}$, the region $\mathcal C \setminus \mathcal P(I)$ has finite (hyper) volume, and consists of the union of finitely many polytopes with apex {\bf 0} whose bases are the compact facets of $\mathcal P(I)$. 
\end{fact}

To set the stage for the proof, let $\ba, \bb$ be two positive integer vectors and $\mathcal H=\{v_1, \ldots, v_h\}$ the Hilbert set for $\mathcal B(\ba, \bb)$. We arrange our indexing of vectors in $\mathcal H$ in a counterclockwise manner, so that we start with $v_1=(1,0)$, order the $v=(r,s)$ via increasing slopes $s/r$, and end with $v_h =(0,1).$ Moreover set ${\bf u}_i= (v_i, \bt(v_i))$, $i=1, \ldots, h$.  Then, in our case, the generators of $\mathcal C$ referenced in Fact \ref{HScalc} are the generators of the cone $C$ in Remark \ref{vectors}.

\medskip
\begin{proposition} \label{sim-vol}
Using the notations above, the simplex $S_i$ generated by ${\bf 0}, {\bolde}_3, \ldots, {\bolde}_{n+2}, {\bf u}_{i}, {\bf u}_{i+1}$ has volume $1/(n+2)!$, for each $i = 1, \dots, h-1$.

\end{proposition}
\begin{proof}[Proof sketch.] 
We apply Fact \ref{Vformula}, taking $v_0 = {\bf 0}$.  Thus,
$$\vol(S_i) = \frac{|\det({\bolde}_3, \ldots, {\bolde}_{n+2}, {\bf u}_{i}, {\bf u}_{i+1})|}{(n+2)!},$$
since the simplex is embedded in an $n+2$-dimensional space.  Next, because ${\bolde}_3, \ldots, {\bolde}_{n+2}$ are standard basis vectors for $\mathbb R^{n+2}$, the calculation of $\det({\bolde}_3, \ldots, {\bolde}_{n+2}, {\bf u}_{i}, {\bf u}_{i+1})$ simplifies to $1^n \cdot \det(v_i,  v_{i+1})$, where $v_i=(r_i,s_i)$ and $v_{i+1}=(r_{i+1},s_{i+1})$ are successive elements in the Hilbert basis.  Finally, $|\det(v_i, v_{i+1})| = 1$ since the elements $v_i$ are ordered in a counterclockwise manner (see, e.g., \cite[Example 7.19]{MS} or Section 2 in~\cite{CDE}).
\end{proof}

\begin{example}  In order to obtain some perspective on the different geometric regions being calculated throughout this paper, see Appendix for a 3-D rendering of the volume for the standard example $\B_R((x),(x))$, where $R = \sfk[x]$.
\end{example}

\begin{theorem} \label{generalHS} Let $\ba, \bb$ be two positive integer vectors and $\mathcal B= \mathcal B(\ba, \bb) = \sfk[Q]$. If $\fm$ is the ideal in $\mathcal B$ generated by the monomials in $Q$, then ${\rm e}(\fm, \mathcal B) = h-1$. Furthermore, if the fan is non-degenerate, ${\rm e}(\fm, \cB)=| \mathcal H_0 | +  | \mathcal H_1 | + \cdots +  | \mathcal H_n | - n-1$.
\end{theorem}

\begin{proof} We can apply Fact \ref{HScalc} to the normal ring $\cB$ and the ideal $\fm$. By Lemma~\ref{group}, $[\mathbb Z^{n+2}: \gp(Q)] = 1$, hence ${\rm e}(\fm, \cB) = (n+2)! \cdot \vol(C \setminus \mathcal P(\fm))$.  Let $\mathcal H = \{v_1, \ldots, v_h\}$, where the elements are ordered in a counterclockwise manner as above. According to~\cite[Theorem 2.8]{M}, $Q$ is generated by the set consisting of $\bu=(v, \bt(v))$, where $\bt(v) = (\max(a_ir, b_is))_{i=1, \ldots, n}$, for $v=(r,s) \in \mathcal H$.

Consider the simplex $S_i$ generated by ${\bf 0}, {\bolde}_3, \ldots, {\bolde}_{n+2}, \bu_{i}, \bu_{i+1}$, for any $i=1, \dots h-1$, which has volume $1/(n+2)!$, by Proposition \ref{sim-vol}. One can see that $S_i \subset C \setminus \mathcal P(\fm)$.  We claim that $S_i$ and $S_{i+1}$ intersect in the plane generated by
${\bf 0}, {\bolde}_3, \ldots, {\bolde}_{n+2}, \bu_{i+1}$. Indeed, if these sets intersect, then there exist nonnegative $\gamma_k, \mu_k$ such that
$\sum_k \gamma_k =1, \sum_k \mu_k =1$ and $$\sum_{k=3}^{n+2}(\gamma_k-\mu_k) {\bolde}_k = \mu_{1}\bu_{i+1}+ \mu_{2} \bu_ {i+2} - \gamma_{1}\bu_{i}-\gamma_{2} \bu_{i+1}.$$
Projecting onto the first two coordinates we get $\mu_2=\gamma_1=0$ and $\mu_1=\gamma_2$, hence the claim follows.  Next, note that $C \setminus \mathcal P(\fm)$ is the union of all simplexes $S_i$, $i=1, \dots, h-1$. Therefore, 
$$\vol (C \setminus \mathcal P(\fm)) = \sum_{i=1}^{h-1}\vol (S_i) = \frac{h-1}{(n+2)!}.$$ 
Thus, ${\rm e}(\fm, \mathcal B)= h-1$.

When the fan is non-degenerate,  $ h =| \mathcal H_0 | +  \cdots +  | \mathcal H_n | - n$ and the statement follows.
\end{proof}

The relation between the Hilbert-Samuel multiplicity and the embedding dimension shows that the intersection algebras of principal monomial ideals have {\it minimal multiplicity}\footnote{See Chapter 2 in~\cite{Sa} for more properties of rings having minimal multiplicity.}: $\nu -d +1= \e$, where $\nu$ is the embedding dimension (see Fact \ref{old}) and $\e$ is the Hilbert-Samuel multiplicity of the ring.  Specifically:

\begin{corollary}  \label{minmult}  Let $\mathcal B = \mathcal B(\ba, \bb)$.The embedding dimension of $\B$ is $h+n$; when the fan is non-degenerate, $\nu(\mathcal B) = | \mathcal H_0 | +  \cdots +  | \mathcal H_n | $.  In general, ${\rm e}(\fm,\mathcal B) = \nu(\mathcal B) - \dim \mathcal B + 1$ or, equivalently, $\e(\fm,\mathcal B) = \nu(\mathcal B) - n-1$.
\end{corollary}

\begin{proof} First of all, $\nu(\B) = h+n$, as it is necessary to add the generators $x_1, \dots, x_n$ to the list of generators of $\m$.  If the fan is non-generate, then $h =  \sum_i h_i - n$, and hence, $\nu(\mathcal B) = | \mathcal H_0 | +  \cdots +  | \mathcal H_n |$.  Since 
$\dim \mathcal B = n+2$, it is straightforward to see that $\rm e(\fm,\mathcal B) = \nu(\mathcal B) - \dim \mathcal B + 1$. 
\end{proof}

We highlight the case where $a_i = b_i$ for all $i = 1, \dots, n$. 

\begin{proposition} \label{HSa=b} Let $n \geq 1$ and $\ba \in \mathbb{Z}^n$, a positive vector. The Hilbert-Samuel multiplicity of $\mathcal B = \mathcal B(\ba, \ba) = \sfk [x_1, \ldots, x_n, A, B, C]/(AB-x_1^{a_1} \cdots x_n^{a_n} C)$ is  $2$.
\end{proposition}

\begin{proof}  As observed in Proposition \ref{q-gorenstein}, the Hilbert set can be easily described in this case. Since each point $(b_i, a_i)$ lies on the line $y=x$, the set is $\mathcal H= \{(1,0), (1,1), (0,1)\}$. Therefore, $\rm e(\fm, \mathcal B) = 3-1$, as per Theorem \ref{generalHS}. Alternatively, the Hilbert-Samuel multiplicity of a hypersurface is its order, which in this case is 2.
\end{proof}

\begin{example} \cite[Example 3.7]{M}  Let $I = (x^5y^2)$ and $J = (x^2y^3)$.  Then the strings $\ba = (5,2)$ and $\bb= (3, 2)$ are fan ordered.
The Hilbert bases are $\mathcal H_0 = \{(0,1), (1,3), (2,5)\}, \mathcal H_1 = \{(1,1), (1,2), (2,5), (3,2)\}$, and $\mathcal H_2 = \{(1,0), (2,1), (3,2)\},$ thus, $\rm e(\fm, \mathcal B) = 7$ and $\nu(\mathcal B) = 10$. The dimension of the algebra is $4$.
\end{example}

%%%%%%%%%%%%%%%%%%%%%%%%%%%%%%%%%%%%%%%%%%%%%%%%%%%%%%%

\section{The $F$-signature}

%%%%%%%%%%%%%%%%%%%%%%%%%%%%%%%%%%%%%%%%%%%%%%%%%%%%%%%

Besides maintaining the notations from Section I, we additionally assume that $\sfk$ is an $F$-finite field of positive characteristic; that is, $[k: k^p] < \infty$.  We are interested in computing the $F$-signature of the intersection algebra, denoted $s(\mathcal B)$. The intersection algebra is an $\mathbb{N}$-graded ring over a field, and therefore one can talk about its $F$-signature; see Chapter II in~\cite{K}.

\begin{fact}[Von Korff, Theorem 3.2.3 in~\cite{K}]
\label{vonkorff}
Let $R$ be a toric ring that does not contain any torus factors. Let ${v}_i$, $i=1, \dots, m$ be its primitive vectors, and let $P_{\sigma}= \{ {\bf u} \in \mathbb{R}^n:  0 \leq \langle {\bf u}, {v}_i \rangle < 1, {\rm for \ all} \ i=1, \ldots, m \}$. Then $s(R) = \vol (P_\sigma)$.
\end{fact}

When $n=1$, we give a formula for the $F$-signature of the ring $\B(\ba, \bb)$.   In the computations, it is necessary to consider two cases, namely when $a$ and $b$ are equal, or not.  When $n > 1$, the situation is much more complicated; we give a general formula for the $F$-signature only in the case that each $a_i = kb_i$ for some positive integer $k$.  The specific assumption $k=1$ forces the intersection algebra to be a binomial hypersurface, as per Proposition \ref{q-gorenstein}, but for $k > 1$ this is not the case, as it can be seen in Proposition~\ref{kb}.

\begin{proposition}  When $n=1$, the $F$-signature of $\cB= \cB(\ba, \bb)$ is
$$
s(\cB)=\begin{cases}
\frac{3b-1}{6ab} + \frac{1}{2a} = \frac{6b-1}{6ab} & \text{ when } a > b \\
 & \\
2\left(\frac{3a-1}{6a^2}\right) = \frac{3a-1}{3a^2}
   & \text{ when } a = b.
\end{cases}
$$
\end{proposition}

\begin{proof} We will apply Fact~\ref{vonkorff} and calculate the volume of $P_{\sigma}$, where the primitive vectors of $\cB$ are ${\bolde}_1=(1,0,0), {\bolde}_2 = (0, 1, 0), \pmba =(-a,0, 1), \pmbb = (0, -b, 1)$, as per Remark \ref{vectors}.  Hence, the $F$-signature is the volume of the region bounded by the inequalities:
$$0 \leq x \leq 1, \qquad 0 \leq y \leq 1, \qquad ax \leq z \leq 1+ax, \quad {\text{ and }} \quad  by \leq z \leq 1 + by.$$
Thus, $\max(ax, by) \leq z \leq \min(1+ax, 1+by)$.  We calculate this volume in two main regions.

Consider the region where $ax \leq by$.  Then $\frac{a}{b}x \leq y \leq \frac{a}{b}x + \frac{1}{b}$ 
and the function $z$ is bounded below by $by$ and above by $1+ax$.  The volume of this region is:

$$\int_0^{\frac{b-1}{a}} \left( \int_{\frac{a}{b}x}^{\frac{a}{b}x+\frac{1}{b}} 1 + ax - by \, dy \right) \, dx  + \int_{\frac{b-1}{a}}^{\frac{b}{a}} \left( \int_{\frac{a}{b}x}^1 1 + ax - by \, dy \right) \, dx = \frac{b-1}{2ab} +  \frac{1}{3ab} = \frac{3b-1}{6ab}.$$

The second region is given $by \leq ax$.  In this case, $\frac{a}{b}x - \frac{1}{b} \leq y \leq \frac{a}{b}x$, and the function $z$ is bounded below by $ax$ and above by $1+by$.

\hskip.25in If $a > b$, then the volume is $\displaystyle{\frac{1}{3ab} + \frac{b-1}{2ab} +  \frac{1}{6ab} = \frac{1}{2a}}$, calculated via

$$\int_0^{\frac{1}{a}} \!\!\left( \int_0^{\frac{a}{b}x} \!\!1 + by - ax \, dy \!\right) dx + \int_{\frac{1}{a}}^{\frac{b}{a}} \!\!\left( \int_{\frac{a}{b}x- \frac{1}{b}}^{\frac{a}{b}x} \!\!1 + by - ax \, dy \!\right) dx + \int_{\frac{b}{a}}^{\frac{b+1}{a}} \!\!\left( \int_{\frac{a}{b}x- \frac{1}{b}}^1 \!\!1 + by - ax \, dy \!\right) dx,$$

\hskip.25in If $a=b$, then the above limit of integration $\frac{b+1}{a}$ exceeds 1, hence instead we compute:

$$\int_0^{\frac{1}{a}} \left( \int_0^{x} 1 + ay - ax \, dy \right) \, dx  + \int_{\frac{1}{a}}^1 \left( \int_{x- \frac{1}{a}}^x 1 + ay - ax \, dy \right) \, dx  = \frac{1}{3a^2} + \frac{a-1}{2a^2} = \frac{3a-1}{6a^2}.$$

The sum of these results give the formul{\ae} in the statement. 
\end{proof} 

Before we state the next result, which allows for $n \geq 1$, we introduce some notation:

Let $n \geq 1$ and $\bb =(b_i: i = 1, \dots, n)$ be a positive integer vector such that $b_1 \geq b_2 \geq \ldots \geq b_n$. Let $S_i(b_1, \ldots, b_n)$ denote the elementary symmetric polynomial of degree $i$ in $b_1, \ldots, b_n$.
% and let $k \geq 1$ be an integer.  WE DON'T NEED k UNTIL MUCH LATER.

Denote $$A = \int_0^{1-1/b_1} \int_0^{1/b_1} \prod_{i=1}^n (1-b_iv) \ dv \ du,$$ 
 $$B = \int_{0}^{1/b_1} \sum_{i=0}^n (-1)^iS_i(b_1,..., b_n)\frac{u^{i+1}}{i+1}  \, du.$$

\begin{lemma}
\label{int}
With the notation introduced above, 
$$A = \left(1- \frac{1}{b_1}\right)\sum_{i=1}^{n} (-1)^{i-1}\frac{S_{i-1}(b_2, \ldots, b_n)}{i(i+1)b_1^{i}}.$$

\end{lemma}
\begin{proof}
Using the notation $S_i(b_1, \ldots, b_n)$ for the degree $i$ symmetric polynomial in $b_1, \ldots, b_n$,
$$A = \int_0^{1-1/b_1} \int_0^{1/b_1} \sum_{i=0}^n (-1)^{i} S_i(b_1, \ldots, b_n) v^i \ dv \ du,$$ hence
\begin{equation*} \label{firstsum}
A= \left(1-\frac{1}{b_1}\right) \sum_{i=0}^n (-1)^i \dfrac{S_i(b_1, \ldots, b_n)}{(i+1)b_1^{i+1}}.
\end{equation*}
Note that for $i \geq 1$ (and $n \geq 1$), $\displaystyle{\frac{S_i(b_1, \ldots, b_n)}{(i+1)b_1^{i+1}} = \frac{S_{i-1}(b_2, \ldots, b_n)}{(i+1)b_1^i} + \frac{S_i(b_2, \ldots, b_n)}{(i+1)b_1^{i+1}}, }$  where, by convention, $S_i(b_2, \ldots, b_n)=0$ for $i=-1$ or $n$. 

So $$\sum_{i=0}^n (-1)^i \dfrac{S_i (b_1, \ldots, b_n)}{ (i+1)b_1^{i+1}} = \sum_{i=0}^n (-1)^i\frac{S_{i-1}(b_2, \ldots, b_n)}{(i+1)a_1^{i}} + \sum_{i=0}^n (-1)^i\frac{S_i(b_2, \ldots, b_n)}{(i+1)b_1^{i+1}},$$ which equals, after reindexing,

$$\sum_{i=-1}^{n-1} (-1)^{i+1}\frac{S_{i}(b_2, \ldots, b_n)}{(i+2)b_1^{i+1}} + \sum_{i=0}^n (-1)^i\frac{S_i(b_2, \ldots, b_n)}{(i+1)b_1^{i+1}}$$
$$= \sum_{i=0}^{n-1} (-1)^i \frac{S_i(b_2, \ldots, b_n)}{(i+1)(i+2)b_1^{i+1}}= \sum_{i=1}^{n} (-1)^{i-1}\frac{S_{i-1}(b_2, \ldots, b_n)}{i(i+1)b_1^{i}}.$$ 
Thus, \centerline{
\xymatrix{
 & *+[F]{\displaystyle{A = \left(1- \frac{1}{b_1}\right)\sum_{i=1}^{n} (-1)^{i-1}\frac{S_{i-1}(b_2, \ldots, b_n)}{i(i+1)b_1^{i}}.}}} }

%$$A = \left(1- \frac{1}{b_1}\right)\sum_{i=1}^{n} (-1)^{i-1}\frac{S_{i-1}(b_2, \ldots, b_n)}{i(i+1)b_1^{i}}.$$
\end{proof}

\begin{remark}
\label{C}
If $b_1 = 1$, then $A = 0$; otherwise, (for $b_1 > 1$) the proof shows that $$\sum_{i=0}^n (-1)^i \dfrac{S_i(b_1, \ldots, b_n)}{(i+1)b_1^{i+1}} = \frac{b_1}{b_1-1} \cdot A.$$

\end{remark}

\begin{lemma} With the notation introduced above,
\label{int1}

$$B = 2 \cdot \sum _{i=1}^{n}  (-1)^{i-1} \frac{S_{i-1}(b_2, \ldots, b_n)}{i(i+1)(i+2)b_1^{i+1}}.$$

\end{lemma}

\begin{proof}
By definition, $\displaystyle{B=\int_{0}^{1/b_1} \sum_{i=0}^n (-1)^iS_i(b_1,..., b_n)\frac{u^{i+1}}{i+1} \ du}$.
Note that for $i \geq 1$ (and $n \geq 1$), 
$\displaystyle{\frac{S_i(b_1, \ldots, b_n)}{(i+1)(i+2)b_1^{i+2}} = \frac{S_{i-1}(b_2, \ldots, b_n)}{(i+1)(i+2)b_1^{i+1}} + \frac{S_i(b_2, \ldots, b_n)}{(i+1)(i+2)b_1^{i+2}}, }$  where, by convention, $S_i(b_2, \ldots, b_n)=0$ for $i=-1$ or $n$. 

Therefore,

$$\sum_{i=0}^n (-1)^i \dfrac{S_i (b_1, \ldots, b_n)}{ (i+1)(i+2)b_1^{i+2}} = \sum_{i=0}^n (-1)^i\frac{S_{i-1}(b_2, \ldots, b_n)}{(i+1)(i+2)b_1^{i+1}} + \sum_{i=0}^n (-1)^i\frac{S_i(b_2, \ldots, b_n)}{(i+1)(i+2)b_1^{i+2}}.$$

This equals, after reindexing, $$\sum_{i=-1}^{n-1} (-1)^{i+1}\frac{S_{i}(b_2, \ldots, b_n)}{(i+2)(i+3)b_1^{i+2}} + \sum_{i=0}^n (-1)^i\frac{S_i(b_2, \ldots, b_n)}{(i+1)(i+2)b_1^{i+2}} = \sum _{i=0}^{n-1}  (-1)^i \frac{2S_{i}(b_2, \ldots, b_n)}{(i+1)(i+2)(i+3)b_1^{i+2}}$$ which gives %$$B = 2 \cdot \sum _{i=1}^{n}  (-1)^{i-1} \frac{S_{i-1}(b_2, \ldots, b_n)}{i(i+1)(i+2)b_1^{i+1}}.$$ 
\centerline{
\xymatrix{
 & *+[F]{\displaystyle{B = 2 \cdot \sum _{i=1}^{n}  (-1)^{i-1} \frac{S_{i-1}(b_2, \ldots, b_n)}{i(i+1)(i+2)b_1^{i+1}}.}}}}
\end{proof}

\begin{theorem} \label{Fsigequaln} Let $n \geq 1$ and $\bb =(b_i: i = 1, \dots, n)$ be a positive integer vector such that $b_1 \geq b_2 \geq \ldots \geq b_n$. For an integer $k \geq 1$, the $F$-signature of $\cB(k \cdot \bb, \bb)$, where $k \cdot \bb = (kb_1, \dots, kb_n)$, equals
\begin{equation}\label{formulak}
\frac{1}{k} \left(\left(\frac{2b_1-1}{b_1-1}\right)A+B \right) \ {\rm when} \ k \geq 2  \ {\rm and} \ b_1 \geq 2,  
\end{equation}

\begin{equation}\label{formulab1}
\frac{1}{k}\left(\sum_{i=0}^n (-1)^i \dfrac{\binom {n}{i}}{i+1}  + B \right) \ {\rm when} \ k \geq 2  \ {\rm and} \ b_1=1, 
\end{equation}

and

\begin{equation}\label{formula1}
2(A+B) \ {\rm when} \ k =1, \ {\rm and} \ {\rm in} \ {\rm particular,} \ 2B  \ {\rm if} \ b_1 = 1.
\end{equation}
\end{theorem}

\begin{proof} Generalizing the argument above, we want to find the volume of the figure made of all of the points $\rho$ such that the inequalities below hold:
$$0 \leq \rho \cdot {\bolde}_1 \leq 1; \quad 0 \leq \rho \cdot {\bolde}_2 \leq 1; \quad 0 \leq \rho \cdot \pmba_i \leq 1; \quad 0 \leq \rho \cdot \pmbb_j \leq 1,$$
where $\alpha_1 = (-kb_1, 0, 1, 0, \dots, 0)$ and $\beta_1 = (0, -b_1, 0, 1, 0, \dots, 0)$.
This translates to finding those $\rho = (x,y,z_1,z_2,\dots, z_n)$ satisfying 
$$0 \leq x \leq 1; \quad 0 \leq y \leq 1; \quad kb_ix \leq z_i \leq 1 + kb_ix; \quad b_iy \leq z_i \leq 1+b_iy.$$

In particular, $\max(kb_ix, b_iy) \leq z_i \leq \min(1+kb_ix, 1+b_iy)$. We calculate this volume in two regions, namely when $kx \leq y$ and $kx \geq y$.

Consider the region where $kx \leq y$.  Then $kx \leq y \leq kx + \frac{1}{b_i}$ for each $i$, 
and hence $kx \leq y \leq kx + \frac{1}{b_1}$ since $b_1 \geq b_2 \geq \cdots \geq b_n$.  Each function $z_j$ is bounded below by $b_jy$ and above by $1+kb_jx$.  The volume of this region is obtained by the sum of the two integrals:
$$
\int_0^{\frac{1}{k}-\frac{1}{kb_1}}\!\!\!\int_{kx}^{kx+\frac{1}{b_1}}\!\!\! \int_{b_ny}^{1+kb_nx} \!\!\!\int_{b_{n-1}y}^{1+kb_{n-1}x} \!\!\!\!\!\!\!\cdots \!\int_{b_2y}^{1+kb_2x} \!\!\!\!\!\!\!\!(1+kb_1x-b_1y) \, dz_2 \, \!\!\cdots \, \!dz_n \, dy \, dx= \!\int_0^{\frac{1}{k}-\frac{1}{kb_1}}\!\!\!\int_{kx}^{kx+\frac{1}{b_1}} \prod_{i=1}^n (1+kb_ix-b_iy) \, dy \, dx$$

and

$$\int_{\frac{1}{k}-\frac{1}{kb_1}}^{\frac{1}{k}}\!\int_{kx}^1 \int_{b_ny}^{1+kb_nx}\!\!\! \int_{b_{n-1}y}^{1+kb_{n-1}x} \!\!\!\cdots \int_{b_2y}^{1+kb_2x} \!\!\!(1+kb_1x-b_1y) \, dz_2 \,\!\! \cdots \, \!dz_n \, dy \, dx = \int_{\frac{1}{k}-\frac{1}{kb_1}}^{\frac{1}{k}}\int_{kx}^1  \prod_{i=1}^n (1+kb_ix-b_iy) \, dy \, dx.$$

Assume $b_1 > 1$.  We will address the case $b_1 = 1$ at the end of the proof.

Let $S' = \{ (x,y) : 0 \leq x \leq \frac{1}{k}-\frac{1}{kb_1}, \ kx \leq y \leq kx+1/b_1 \}.$  For the first integral, we will compute 
$$J_{k} = \iint_{S'} \prod _{i=1}^n (1-b_i(y-kx)) \ dy \ dx.$$
Consider the substitution:
$$ x= g(u,v)= u/k, \quad y=h(u,v)=u+v,$$ 
with $S = \{ (u,v) :  0 \leq u \leq 1-1/b_1, \ 0 \leq v \leq 1/b_1 \}.$  Note $S' = f(S),$ where $f=(g,h): S \to S'$.  Therefore 
$$J_k = \iint_S \prod_{i=1}^n (1-b_iv) \ dv \ du = \frac{1}{k} \cdot \int_0^{1-1/b_1} \int_0^{1/b_1} \prod_{i=1}^n (1-b_iv) \ dv \ du = \frac{1}{k} A.$$

%This is what appears in {\it one half of the first half} of the formula \eqref{formula}.

Now let us go to the second integral:
$$J'_k= \int_{\frac{1}{k}-\frac{1}{kb_1}}^{\frac{1}{k}}\int_{kx}^1  \prod_{i=1}^n (1-b_i(y-kx)) \, dy \, dx.$$
Let $S' = \{ (x, y) : \frac{1}{k}- \frac{1}{kb_1} \leq x \leq \frac{1}{k}, \ kx \leq y  \leq 1 \}.$  We will make the following substitution:
$$ x = g(u,v)= u/k, \quad y=h(u,v)=u+v,$$ 
with $S = \{ (u,v) :  1-1/b_1\leq u\leq 1, \ 0 \leq v \leq 1-u \}.$  Note $S' = f(S),$ where $f=(g,h): S \to S'$.  After the substitution we get
\begin{equation}\label{J'}
J'_k =\frac{1}{k} \int_S \prod_{i=1}^n (1-b_i v) \ dv \ du = \frac{1}{k}\int_{1-1/b_1}^{1} \int _0^{1-u}  \prod_{i=1}^n (1-b_i v) \ dv \  du.
\end{equation}
This equals
$$\frac{1}{k}\int_{1-1/b_1}^{1}  \sum_{i=0}^n (-1)^i S_i (b_1, \ldots, b_n) \frac{(1-u)^{i+1}}{i+1} \ du,$$ which, after a simple substitution, gives $\displaystyle{J'_k=\frac{1}{k}B}.$

%\begin{equation} \label{secondsum}
%J'_k=\frac{1}{k}B.
%\end{equation}

When $y \leq kx$ and $k \geq 2$,  we have $y + \frac{1}{b_1} \leq 2$ and so $y \leq kx \leq y + \frac{1}{b_1}$ automatically gives $y \leq 1$, so there is only one integral to compute:

%\textcolor{red}{Comment: It seems that this integral is only in play when $k = 2$ (and $b_1 = 1$), else the region is empty.  True?}

\begin{equation}\label{kbig}
\int_0^{1}\!\!\!\int_{\frac{y}{k}}^{\frac{y}{k}+\frac{1}{kb_1}}  \prod_{i=1}^n (1+b_iy-kb_ix) \, dx \, dy.
\end{equation}
Set $S' = \{ (x, y) : \frac{y}{k} \leq x \leq \frac{y}{k} + \frac{1}{kb_1}, \ 0 \leq y \leq 1 \}$, and make the substitution $y = h(u,v)= u, x= g(u,v)=\frac{u+v}{k},$
with $S = \{ (u,v) :  \ 0 \leq u \leq 1, 0 \leq v \leq 1/b_1 \}.$  Note $S' = f(S),$ where $f=(g,h): S \to S'$.  We obtain
$$ \frac{1}{k} \cdot \int_0^{1}\!\!\!\int_0^{\frac{1}{b_1}}  \prod_{i=1}^n (1-b_iv) \, dv \, du = \frac{1}{k}\sum_{i=0}^n (-1)^i S_i(b_1, \ldots, b_n) \frac{1}{b_1^{i+1}(i+1)},$$ which equals $$\frac{b_1}{k(b_1-1)} \cdot A,$$ according to Remark~\ref{C}.

When $y \leq kx$ and $k=1$, then there are two integrals to compute:

\begin{equation} \label{k=1p1}
\int_0^{1-\frac{1}{b_1}}\!\!\!\int_{y}^{y+\frac{1}{b_1}}  \prod_{i=1}^n (1+b_iy-b_ix) \, dx \, dy
\end{equation}
and

\begin{equation} \label{k=1p2}
\int_{1-\frac{1}{b_1}}^{1}\!\!\!\int_{y}^{1}  \prod_{i=1}^n (1+b_iy-b_ix) \, dx \, dy.
\end{equation}

The first integral uses $S' = \{ (x, y) : y \leq x \leq y +1 /b_1, \ 0 \leq y \leq 1-1/b_1  \}$, with the substitution $y = h(u,v)= u, x=g(u,v)=u+v,$
with $S = \{ (u,v) :  \ 0 \leq v \leq 1/b_1, 0 \leq u \leq 1-1/b_1  \}.$  Again, $S' = f(S),$ where $f=(g,h): S \to S'$.  We obtain
$$\int_0^{1-\frac{1}{b_1}}\int_{0}^{1/b_1}  \prod_{i=1}^n (1-b_iv) \, dv \, du \  = A.$$

The last integral uses $S' = \{ (x, y) : y \leq x \leq 1, \ 1-1/b_1 \leq y \leq 1 \}$, with the substitution $y = h(u,v)= u, x=g(u,v)=u+v,$
with $S = \{ (u,v) :  \ 0 \leq v \leq 1-u, 1-1/b_1 \leq u \leq 1 \}.$  Again, $S' = f(S),$ where $f=(g,h): S \to S'$.  We obtain
$$\int_{1-\frac{1}{b_1}}^1\int_{0}^{1-u}  \prod_{i=1}^n (1-b_iv) \, dv \, du,$$
which is $k$ times the last double integral in Equation \eqref{J'}; i.e., this expression is equal to $B$.

%This equals 
%$$\int_{1-\frac{1}{b_1}}^1 \sum_{i=0}^n (-1)^iS_i(b_1,..., b_n)\frac{(1-u)^{i+1}}{i+1}  \, du.$$
%A simple substitution gives
%$$\int_{0}^{1/b_1} \sum_{i=0}^n (-1)^iS_i(b_1,..., b_n)\frac{u^{i+1}}{i+1}  \, du,$$which is 
%$$\sum_{i=0}^n (-1)^iS_i(b_1,..., b_n)\frac{(1/b_1)^{i+2}}{(i+1)(i+2)}=B.$$

Therefore, when $k \geq 2$ (and $b_1 > 1$), one has $\displaystyle{\frac{1}{k}A + \frac{1}{k}B+ \frac{b_1}{k(b_1-1)}A}$, which simplifies to  formula \eqref{formulak}, and when $k = 1$ formula \eqref{formula1} is easily obtained.

Now assume that $b_1 = 1$.  Then the integral $J_k = 0$, and in particular $A = 0$, while $J_k'$, with limits of integration  0 and $\frac{1}{k}$ on $x$, still equals $\frac{1}{k}B$.  When $k \geq 2$, the integral in \eqref{kbig} is 
$$ \int_0^{1}\!\!\!\int_{\frac{y}{k}}^{\frac{y}{k}+\frac{1}{k}}  \prod_{i=1}^n (1+y-kx) \, dx \, dy = 
\sum_{i=0}^n (-1)^i \dfrac{S_i(1, \ldots, 1)}{i+1} = \sum_{i=0}^n (-1)^i \dfrac{\binom {n}{i}}{i+1},$$
while the integrals in \eqref{k=1p1} and \eqref{k=1p2} are clearly 0. The formul{\ae} now follow.

\end{proof}
\excise{
\begin{corollary}
The F-signature of $\cB (\bb, \bb)$ equals

$$2(1- \frac{1}{b_1})\sum_{i=1}^{n} (-1)^{i-1}\frac{S_{i-1}(b_2, \ldots, b_n)}{i(i+1)b_1^{i}} + 2 \cdot \sum _{k=1}^{n}  (-1)^{k-1} \frac{S_{k-1}(b_2, \ldots, b_n)}{k(k+1)(k+2)b_1^{k+1}}.$$

\end{corollary}

\begin{proof}
This can be obtained form the earlier result by letting $k=1$ in Theorem~\ref{Fsigequaln}, and then applying Propositions~\ref{int} and \ref{int1}.

The formula obtained can be deduced from Conca's work (see Theorem 3.1 in~\cite{Co}) which gives a formula for the Hilbert-Kunz multiplicity of binomial hypersurfaces. One can note that
$$ \cB(\bb, \bb) = \sfk [x_1, \ldots, x_n, A, B, C]/(AB-x_1^{b_1} \cdots x_n^{b_n} C),$$which is a hypersurface of multiplicity 2 and hence its F-signature and Hilbert-Kunz multiplicity add up to $2$.

\end{proof}

}

\begin{remark} 
\begin{enumerate}
\item
The hypothesis $b_1 \geq b_2 \geq \ldots \geq b_n$ is nonrestrictive, since we can always reindex the indeterminates $x_1, \ldots, x_n$ to obtain this condition for the vector $\bb$.
\item
Not surprisingly, the denominators in these terms arise from polygonal sequences: those in the first term are the triangular numbers, i.e., $n(n+1)/2$, while the denominators in the second term are the $n$-th $n$-gonal numbers minus $n$; i.e., $(n(n-1)(n-2))/2$.
\item 
As an example, the intersection algebra $\cB(3,2)$ has $F$-signature equal to $11/36$.
\item
For $n=2$ and $b_1 \geq b_2 >0$, the $F$-signature of $\cB ((b_1, b_2), (b_1, b_2))$ is

$$2\left(1-\frac{1}{b_1}\right)\left(\frac{1}{2b_1} - \frac{b_2}{6b_1^2}\right) + 4\left(\frac{1}{6b_1^2} - \frac{a_2}{24b_1^3}\right) = \frac{6b_1^2-2b_1-2b_1b_2 +b_2}{6b_1^3}.$$

\end{enumerate}
\end{remark}

%%%%%%%%%%%%%%%%%%%%%%%%%%%%%%%%%%%%%%%%%%%%%%%%%%%%%%%

\section{The Hilbert-Kunz Multiplicity}

%%%%%%%%%%%%%%%%%%%%%%%%%%%%%%%%%%%%%%%%%%%%%%%%%%%%%%%

Let $\sfk$ be a field of positive characteristic. The purpose of this section is to present formul{\ae} for the Hilbert-Kunz multiplicity of the intersection algebra of principal monomial ideals in some interesting cases. The Hilbert-Kunz multiplicity is notoriously difficult to compute, even for hypersurfaces in three variables. General formul{\ae} for classes of rings are seldom available. In the normal semigroup ring case, a general approach is already known due to K.-i. Watanabe \cite{W} and subsequently, E.~Kazufumi \cite{Kz}, which reduces the computation of the Hilbert-Kunz multiplicity to that of a relative volume of a certain region. 
In the case of our rings $\mathcal B (\ba, \bb)$, we will present some computations that, besides providing clean formul{\ae}, show that general formul{\ae} for $\hk(B(\ba, \bb))$ in terms of $\ba, \bb$ can be very difficult to obtain. Moreover, in the Appendix, we give an example to show how one can calculate the Hilbert-Kunz multiplicity with the use of {\it Mathematica} when $n=1$ and the values of $\ba = a, \bb = b$ are specified. 

Again, because our objects of study are normal semigroup rings, some of the existing literature addresses our situation.  For example, by \cite[Theorem 2.1]{W}, the Hilbert-Kunz multiplicity of the intersection algebra that we aim to compute is known to be a rational number.  More importantly, the exact value can be obtained through calculating the volume of a particular region:

\begin{fact} \label{K-HKthm} \cite[Theorem 2.2]{Kz} Let $Q$ be a semigroup and $q_1, \dots, q_v \in Q \subset \mathbb Z^N$ elements such that $\sfk[Q]/\fq$ is finite length, where $\fq = (x^{q_1}, \dots, x^{q_v})$.  Let $\mathcal C$ denote the convex rational polyhedral cone spanned by $Q$ in $\mathbb R^N$ and $\mathcal P = \{p \in \mathcal C | p \notin q_j+\mathcal C {\text{ for each }} j\}$.  Then $\e_{HK}(\fq, \sfk[Q]) = \rvol \bar{\mathcal P}$, where $ \bar{\mathcal P}$ is the closure of $\mathcal P$ and $\rvol$ denotes the relative volume, as per \cite[p.~239]{St}.
\end{fact}

Recall that $\B (\ba, \bb) = \sfk [Q]$.  Index the vectors in the Hilbert set $\mathcal H = \{v_1, \ldots, v_h\}$ in a counterclockwise manner, as described in Section 3, with $v_1= (1,0)$ and $v_h =(0, 1)$. As mentioned previously, the generators of $Q$ are $\bu= \bu(v)=(v, \bt(v))$, where $\bt(v) = (\max(a_ir, b_is))_{i=1, \ldots, n}$, for $v=(r,s) \in \mathcal H$. Let $\mathcal G = \{ \bu(v) : v \in \mathcal H\}$. 
As in Section 3, the generators of $\mathcal C$ referenced in Fact \ref{K-HKthm} are the generators of the cone $C$ in Remark \ref{vectors}.
Set $v=(r,s)$ and let $C_v $ denote the complement of the translation of the cone $C$ by $\bu(v)$; that is, the complement of $\bu(v) +C$. In other words, 
$$C_v=\{ (x, y, t_1, \ldots, t_n) : x < r  \ {\rm or} \  y < s \ {\rm or} \ t_i -\max(a_ir, b_i s) < \min( a_i(x-r), b_i (y-s)), i =1, \ldots, n\}.$$

Therefore, regarding the intersection algebra, Fact \ref{K-HKthm} is applied and interpreted as the following:

\begin{proposition}
\label{ehk}

The Hilbert-Kunz multiplicity of $\mathcal B(\ba, \bb)$ equals the relative volume of $C \cap (\cap_{v \in \mathcal H} C_v).$
\end{proposition}

The first case where we compute the Hilbert-Kunz multiplicity essentially comes from knowing the $F$-signature.  The result below provides the connection between the two invariants.

\begin{proposition}[See Huneke-Leuschke, \cite{HL}]
\label{HunLe}
Let $(R, \fm)$ be an $F$-finite prime characteristic local ring or an $\mathbb{N}$-graded ring over a field, with homogeneous maximal ideal $\fm$. Assume that $R$ is Gorenstein of multiplicity 2.  Then
$$\hk(R) + s(R)=2.$$
\end{proposition}

\begin{proof}
Let $I$ be an ideal generated by a system of parameters. Let $u \in R$ be the socle generator for $R$. According to~\cite{HL}, proof of Lemma 12, $\hk(I) - \hk (I, u) = s(R).$ We can assume that $I$ is generated by a minimal reduction for $\fm$ (since enlarging the residue field does not change the invariants in the statement), and hence $2=\e(R)= \e(I) =\hk(I).$ Therefore, it is enough to show that $(I, u) =\fm$.  Since $\e(R)=\e(I)={\rm length}_R(R/I)$, we have that either $I= (I: \fm)$ or $(I: \fm) =\fm$. The first equality can not happen, because $R$ is not regular, and hence $(I, u) = (I: \fm) =\fm$.

\end{proof}

\begin{corollary} \label{a=b} If $a_i = b_i$ for all $i = 1, \dots, n$, then $\hk(\mathcal B) = 2 - s(\mathcal B).$  In particular:
\begin{enumerate}
\item For $n =1$ and $a=b$,
$$\hk(\mathcal B(a, a)) = 2 - \frac{3a-1}{3a^2}= \frac{1-3a+6a^2}{3a^2};$$

\item In general (recalling that $S_k(a_1, \ldots, a_n)$ denotes the $k$-symmetric polynomial in $a_1, \ldots, a_n$),
$$\hk(\mathcal B(\ba, \ba)) = 2 - 2\left(1-\frac{1}{a_1}\right) \sum_{k=1}^n (-1)^{k+1} \dfrac{S_{k-1}(a_2, \ldots, a_n)}{k(k+1)a_1^k} - 4 \cdot \!\!\sum_{k=1}^n (-1)^{k+1} \dfrac{S_{k-1}(a_2, \ldots, a_n)}{k(k+1)(k+2)a_1^{k+1}}.$$
\end{enumerate}
\end{corollary}

\begin{proof}The Hilbert-Kunz multiplicity does not change if we enlarge the field $\sfk$ to its algebraic closure. So we can assume that $\sfk$ is $F$-finite and hence the intersection algebra is also $F$-finite.
When $a_i=b_i$ for all $i = 1, \dots, n$, the ring $\mathcal B = \mathcal B(\ba, \ba)$ is a hypersurface, according to Proposition~\ref{q-gorenstein}, of multiplicity $2$. By Proposition~\ref{HunLe}, it is known that $\hk(\mathcal B) + s(\mathcal B) = 2$. Since we computed the $F$-signature in this case, the result follows.  
\end{proof}

The second case that we consider is $\cB(a, b) = \cB(kb, b)$, where $k$ and $b$ are both positive integers and $a = kb$.  Before calculating the Hilbert-Kunz multiplicity, we provide an explicit description of the intersection algebra $\cB(kb, b)$.

%%We corrected some typos in the matrix and proof below.

\begin{proposition}
\label{kb} 
For $n=1$, if $a =k b$ for some $k \in \mathbb N_+$, $b \in \mathbb{Z} \setminus \{0\}$, then
$$\cB(kb, b) \simeq \dfrac{\sfk[x_0, \ldots, x_{k+2}]}{I_2(M_{k,b})},$$ where
$$M_{k,b} = \left(\begin{array}{ccccc}
x_0^b & x_1 & x_3& \cdots & x_{k+1} \\
x_2 & x_3 & x_4 &\cdots & x_{k+2}
\end{array}
\right)$$

\end{proposition}

\begin{proof}
This is an immediate application of Theorem 3.4 in~\cite{CDE}.  The Hilbert set is 
$$\mathcal H = \{(1,0), (0,1), (1, k), (1, k-1), \dots, (1, 2), (1, 1)\},$$  
and the set $\mathcal G$ determined by $\mathcal H$, as per Definition \ref{Hilbertsets}, is shown below, where we also include the exponent vector corresponding to $x$:
$$\mathcal G = \{(0,0,1), (1,0, kb), (0,1,b), (1, k, kb), (1, k-1, kb), \dots, (1, 2, kb), (1, 1, kb)\}.$$
The generators of the presentation ideal can be listed since the Hilbert set of the intersection algebra is known. The Hilbert number is $k+2$, so there will be $\binom{k+1}{2}$ relations among the generators $x, x^{kb}u, x^bv, x^{kb}uv^k, \ldots, x^{kb}uv$, which correspond to the vectors in $\mathcal G$ above, over the field $\sfk$. These are
$$ x^{kb}uv^{i} \cdot x^{kb}uv^{j} = (x^{kb} uv^{(i+j)/2})^2,$$ if $i, j$ are nonconsecutive and of the same parity, and
$$ x^{kb}uv^{i} \cdot x^{kb}uv^{j}= x^{kb} uv^{(i+j-1)/2} \cdot x^{kb} uv^{(i+j+1)/2},$$ if $i, j$ are nonconsecutive and of different parity. We also have
$$x^b \cdot x^{kb}uv^{i} = x^{kb}uv^{i-1}\cdot x^bv, \quad {\text{for }} i=2, \ldots, n.$$

If we denote $x_0 = x, x_1=x^{kb}u, x_2=x^bv, x_3 = x^{kb}uv, \ldots, x_{k+2}=x^{kb}uv^k$, then it is a simple exercise to see that the ideal generated by the relations above equals $I_2(M_{k,b})$.
\end{proof}

\begin{corollary} (See Watanabe \& Yoshida, \cite[Example 3.5]{WY1}.) \label{WYscroll} Let $a$ be a positive integer. The intersection algebra $\mathcal B (a, 1)$ is isomorphic to the rational normal scroll 
$$S = \sfk[T, xT, xyT, yT, x^{-1}yT, x^{-2}yT, \dots, x^{-(a-1)}yT].$$ 
\end{corollary}

\begin{proof}  According to Theorem 3.2.1 in \cite{M-thesis} or Proposition~\ref{kb}, for $a \geq 1$, $\mathcal B (a,1) \simeq \sfk [x_1, \ldots, x_{a+3}]/I_2(M)$, where $I_2(M)$ represents the ideal generated by the $2 \times 2$ minors of the matrix
\[
M= \left(
\begin{array}{ccccc}
x_0 & x_1 & x_3& \cdots & x_{a+1} \\
x_2 & x_3 & x_4 &\cdots & x_{a+2}
\end{array}
\right).
\]
It is known that this is a presentation of the rational normal scroll under $x_0=xT, x_1=T, x_2=xyT, x_3=yT, x_4=x^{-1}yT, \ldots, x_{a+3}=x^{-(a-1)}yT.$
\end{proof}

We are now ready to calculate the Hilbert-Kunz multiplicity of $\mathcal B (kb, b)$.

\begin{theorem} \label{HKformula} For $n=1$, if $a =k b$ for some $k \in \mathbb N_+$, $b \in \mathbb{Z} \setminus \{0\}$, then $\hk(\fm, \mathcal B (kb, b))$ is 
\begin{equation} \label{HKeqn}
\frac{(k+1)-6kb+3k(k+3)b^2}{6kb^2} = \frac{a+b-6ab+3a^2b+9ab^2}{6ab^2}.
\end{equation}
\end{theorem}

\begin{proof} Since Corollary \ref{a=b} takes care of the case $k = 1$ and the case $b = 1$ is handled in \cite[Example 3.5]{WY1} via Corollary \ref{WYscroll}, we concentrate on the situation when $b, k$ are both greater than $1$.  

Using Fact \ref{K-HKthm}, the volume of the three dimensional region given by the inequalities below needs to be computed:

I. $x < 1$ or $z \leq by + kb$; \qquad II. $y < 1$ or $z \leq bkx + b$; \qquad III. $x < 1$ or $y < k$;  \qquad and

IV. $x < 1$ or $y < ki$ (for $1 \leq i \leq k$) or $z \leq by + b(k-i)$.

However, these inequalities are not mutually exclusive conditions.  A partition of the polytope above is given by the subdivision below:

\begin{enumerate}
\item[1.] $\{x < 1, \quad z \leq bkx + 1 \}$
\item[2.] $\{x < 1, \quad y > 1,  \quad z > bkx + 1, \quad z \leq by + 1, \quad z \leq bkx + b \}$
\item[3.] $\{x < 1, \quad y < 1, \quad z > bkx + 1, \quad z \leq by + 1\}$
\item[4.] $\{x > 1, \quad y < 1, \quad z \leq by + 1 \}$
\item[5.] $\{x > 1, \quad 1 \leq y < k, \quad z \leq by + 1, \quad z \leq bkx + b \}$
\item[6.] $\{x > 1, \quad 0 \leq y < 1, \quad z > by + 1, \quad z \leq bkx + 1, \quad z \leq by + b(k-0) \}$
\item[7.] $\{x > 1, \quad 1 \leq y < 2, \quad z > by + 1, \quad z \leq bkx + 1, \quad z \leq by + b(k-1) \}$
\item[8.] $\{x > 1, \quad 2 \leq y < 3, \quad z > by + 1, \quad z \leq bkx + 1, \quad z \leq by + b(k-2) \}$
\item[$\vdots$] \hskip2in $\vdots$
\item[$k+$5.] $\{x > 1, \quad (k-1) \leq y < k, \quad z > by + 1, \quad z \leq bkx + 1, \quad z \leq by + b(k-(k-1)) \}$.
\end{enumerate}

We calculate the volumes of each of the regions above, respectively, recalling that these formul{\ae} hold for $b, k > 1$ only. 
$$\frac{1+bk}{2b}, \hskip.08in \frac{(2k-1)(-1+b)}{2bk}, \hskip.08in \frac{1-3b+3b^2}{6b^2k}, \hskip.08in 0, \hskip.08in \frac{1}{6b^2k}, \hskip.08in \frac{1-3b+3b^2}{6b^2k}, \dots, \frac{1-3b+3b^2}{6b^2k}, \hskip.08in \frac{-1+b}{2bk}.$$
The expression $\displaystyle{\frac{1-3b+3b^2}{6b^2k}}$ appears $k$ times, therefore, the total volume is:
$$\frac{1+bk}{2b} + \frac{(2k-1)(-1+b)}{2bk} + \frac{k(1-3b+3b^2)}{6b^2k} + 0 + \frac{1}{6b^2k} + \frac{-1+b}{2bk}=\frac{(k+1)-6kb+3k(k+3)b^2}{6kb^2}.$$ 
\end{proof}

\begin{corollary}  The formula for $\hk(\fm, \mathcal B(a, 1))$ shown in \eqref{HKeqn} agrees with that of Watanabe and Yoshida in \cite[Example 3.5]{WY1}. 
\end{corollary}

\begin{proof}
With $\fm_S$ the irrelevant maximal ideal of the rational normal scroll $S$ in Corollary \ref{WYscroll}, Watanabe and Yoshida \cite[Example 3.5]{WY1} show that the Hilbert-Kunz multiplicity of $S_{\fm_S}$ is given by the formula
\begin{equation} \label{scroll}
\hk(S_{\fm_S}) = \frac{\rm e(S_{\fm_S})}{2} + \frac{\rm e(S_{\fm_S})}{6a},
\end{equation}
where ${\rm{e}}(S_{\fm_S}) = a+1$.  Making these substitutions in Eq.~\eqref{scroll}, we obtain
$$\frac{1+4a+3a^2}{6a},$$
which is exactly the formula for $\hk(\fm, \mathcal B)$ in Theorem \ref{HKformula} with $b = 1$ and $a = kb$.
\end{proof}

\begin{remark}  The formula for $\hk(\fm, \mathcal B(a, a))$ shown in Corollary \ref{a=b} {\it (1)} can be obtained by taking $k=1$ in Theorem~\ref{HKformula}.
\end{remark}

%%%%%%%%%%%%%%%%%%%%%%%%%%%%%%%%%%%%%%%%%%%%%%%%%%%%%%%

\section{Appendix}

%%%%%%%%%%%%%%%%%%%%%%%%%%%%%%%%%%%%%%%%%%%%%%%%%%%%%%%

%~~~~~~~~~~~~~~~~~~~~~~~~~~~~~~~~~
\subsection{Code for Computation}
%~~~~~~~~~~~~~~~~~~~~~~~~~~~~~~~~~

To find the Hilbert-Kunz multiplicity using computer software for a specific example $\mathcal B(a, b)$ in the case when $n=1$, with $a \geq b$, one should use an adaptation of the code below.  

\begin{example}
Assume $a=3, b=2$. Then $\mathcal H =\{(1,0), (0,1), (2,3), (1,1), (1,2) \}$, which can be obtained with {\it Macaulay 2}, as per the code \cite[p.~631]{M}.  Next, using the {\it Mathematica} code below for computation, we get that the Hilbert-Kunz multiplicity of $\cB(3,2)$ at its homogeneous maximal ideal is $41/18$.  (Recall from Section 3 that $s(\cB(3,2)) = 11/36$.)

{\tt Integrate[Boole[((x >= 0 \&\& y >= 0 \&\& z >= 2y \&\& z >= 3x)) }

{\tt \&\& ((z < 2y + 1 ||  z < 3x + 1))  \&\& ((x < 1 || z < 2y + 3)) \&\&  ((y < 1 || z < 3x + 2))}

{\tt \&\& ((x < 2 || y < 3)) \&\& ((x < 1 || y < 1 || z < 2y + 1)) }

{\tt \&\&  ((x < 1 || y < 2 || z < 3x + 1))],  \{x, 0, 10\}, \{y, 0, 10\}, \{z, 0, 20\} ]}  

\end{example}

%~~~~~~~~~~~~~~~~~~~~~~~~~~~~~~~~~~~~~~~~~~~~~~~~~~~
\subsection{Three Dimensional Regions-A Standard Example}
%~~~~~~~~~~~~~~~~~~~~~~~~~~~~~~~~~~~~~~~~~~~~~~~~~~~

The following graphs show the three distinct volumes (which equal 2, 2/3, and 4/3, respectively) that are calculated in the previous three sections, for the simplest, and most familiar, example of an intersection algebra; i.e., for $I = (x)$, $J = (x)$, in $R=k[x]$.  The ring $\B_R(I,J)$ is generated over $\sfk$ by $x, xu, xv, xuv$ and $\B \cong \sfk[x_1,A,B,C]/(AB-Cx_1)$, as per Proposition \ref{q-gorenstein}.

\begin{figure}[tbh]
    \showtwo{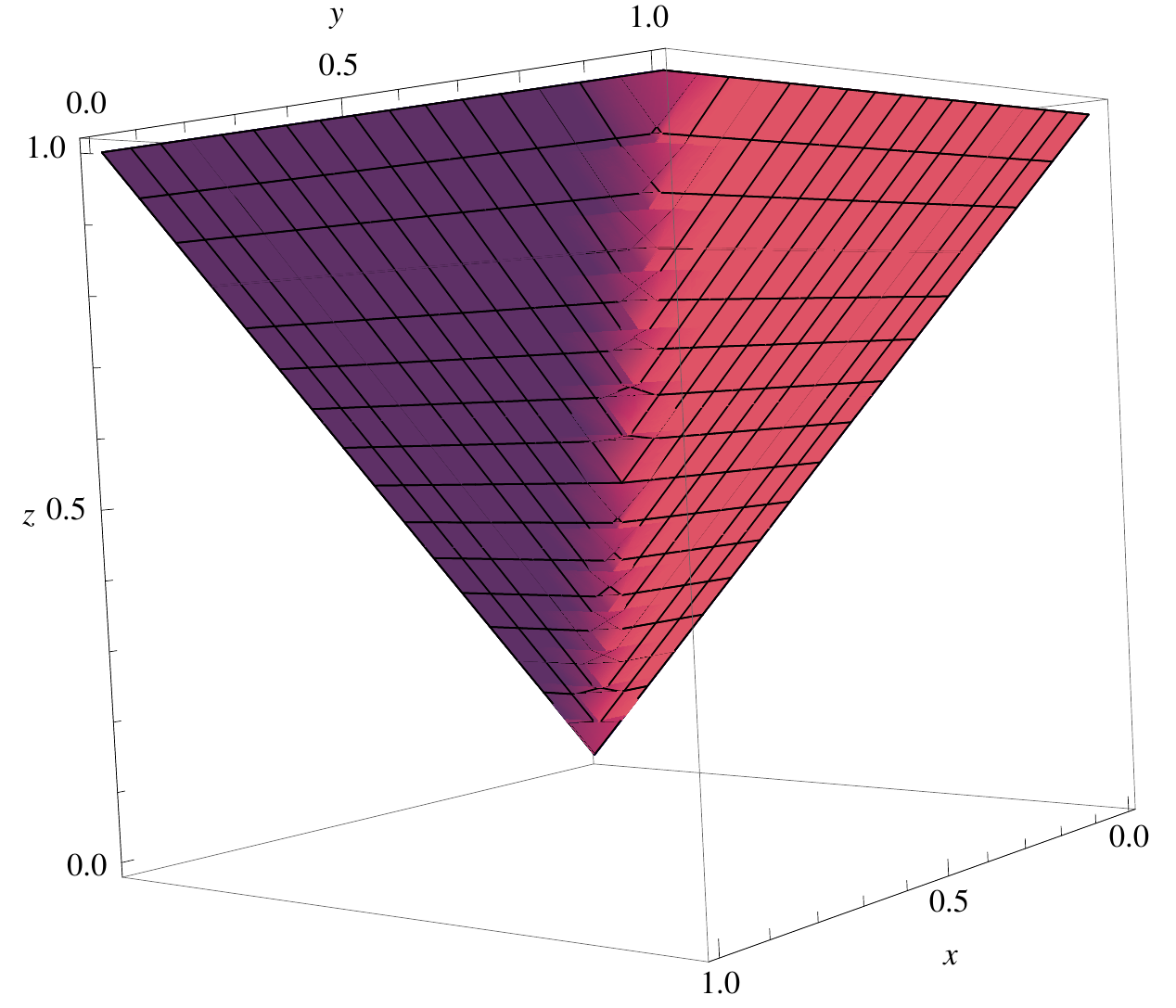}{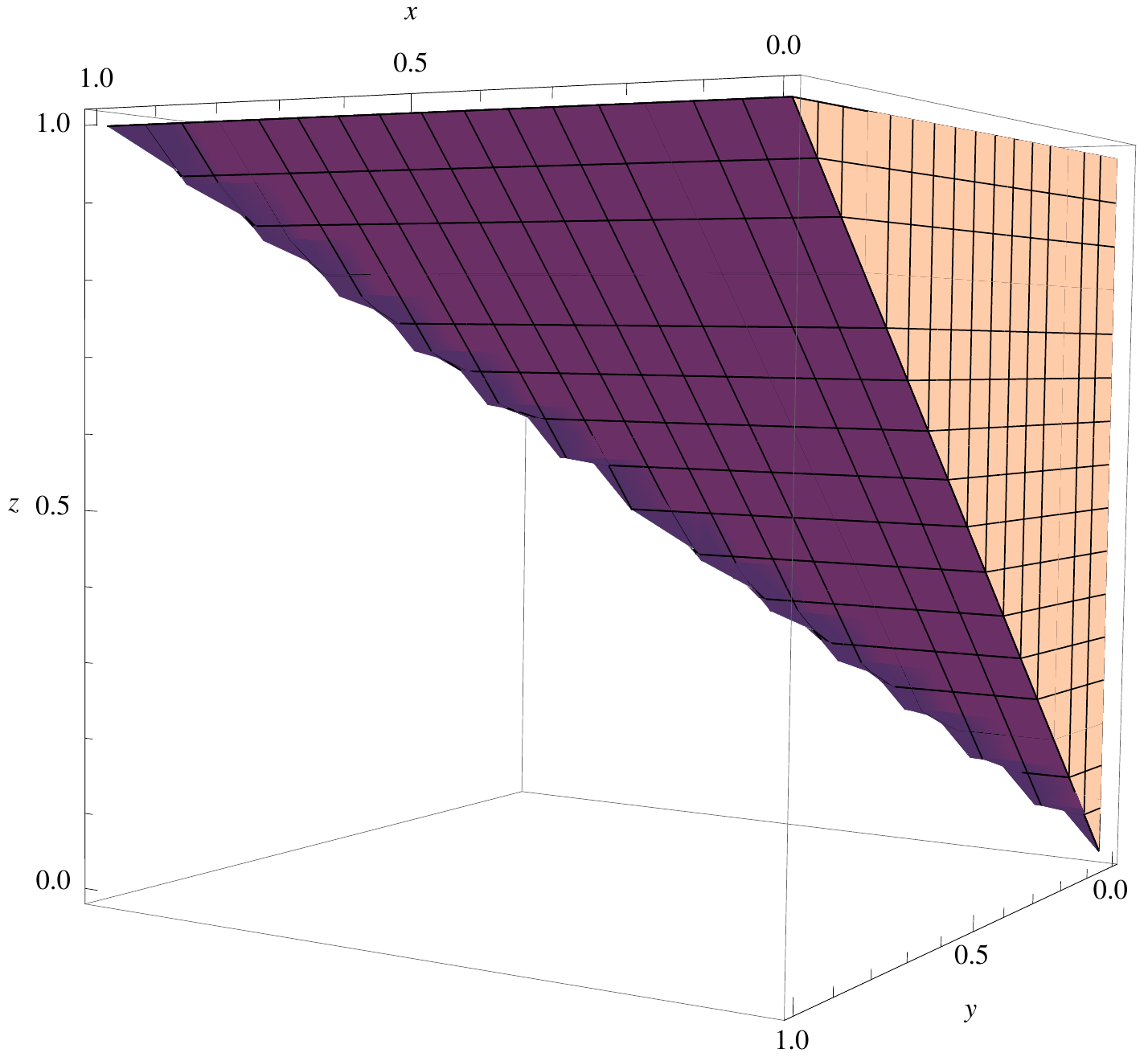}{0.25}
    \caption{{\bf Hilbert-Samuel Multiplicity}: view from the front and from the side after rotation around the $z$-axis.}
\end{figure}

\begin{figure}[tbh]
   \showtwo{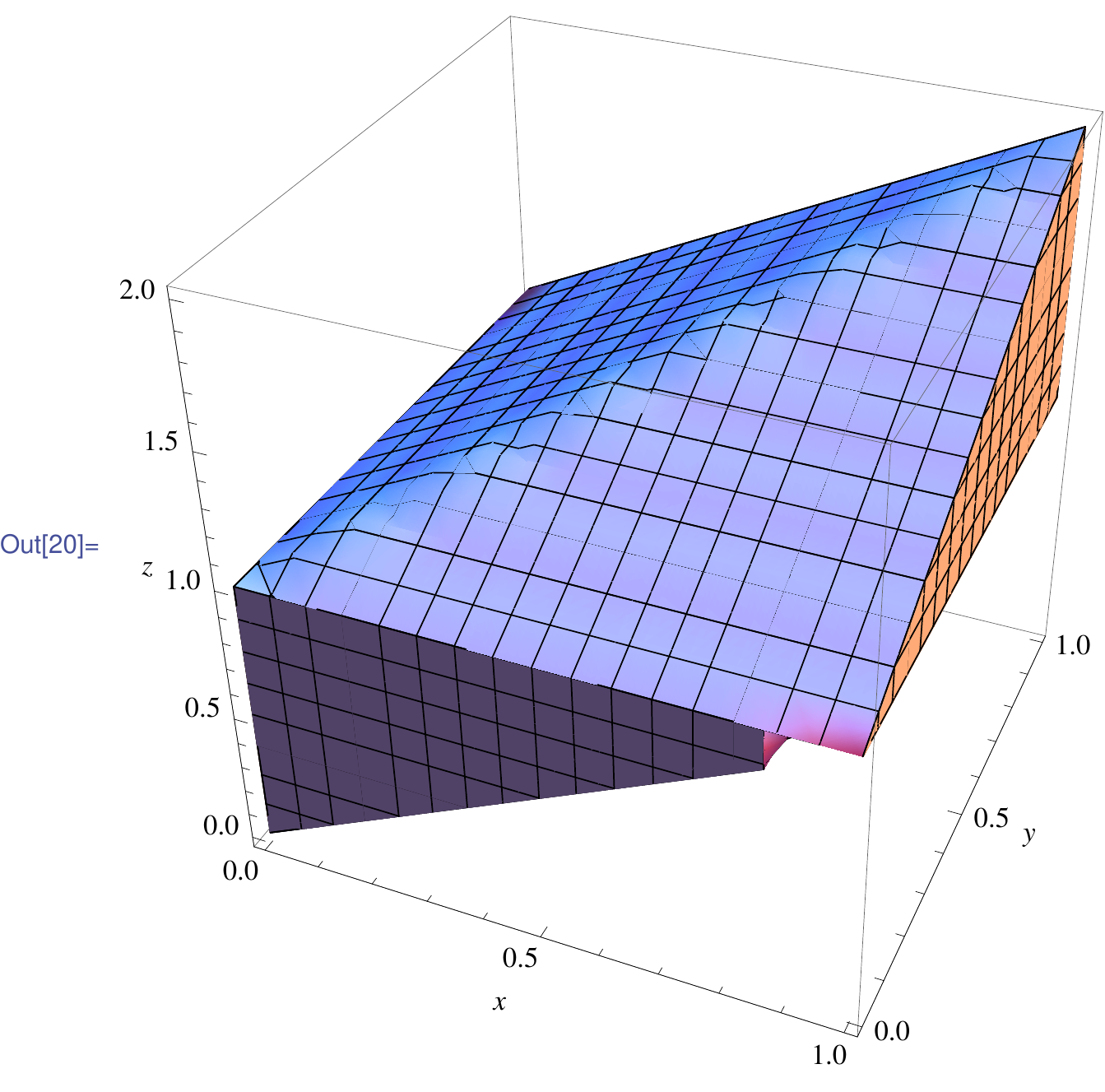}{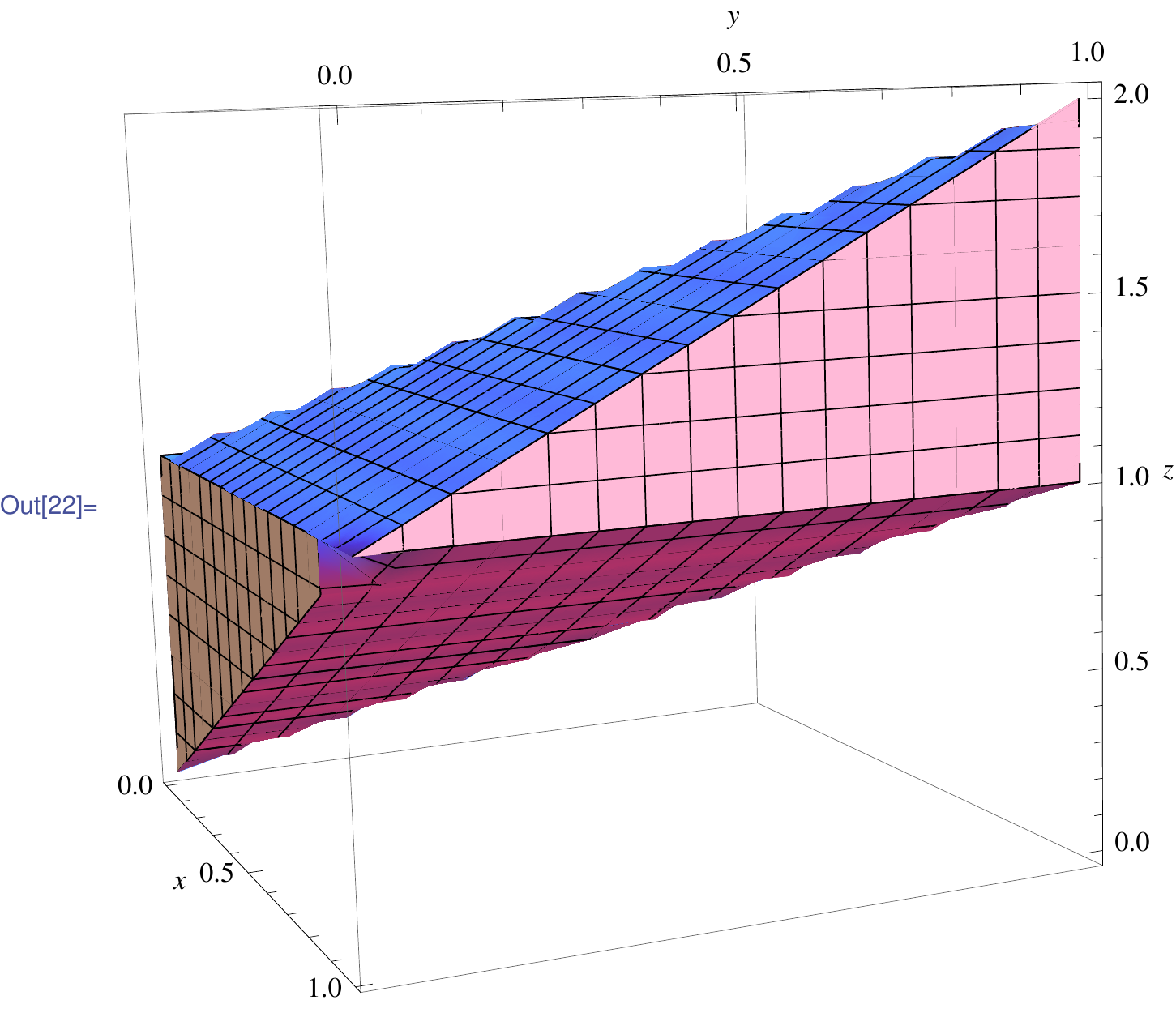}{0.25}
    \caption{{\bf $F$-Signature}: view from above and from the side after rotation around the $z$-axis.}
 \end{figure}

 \begin{figure}[tbh]
    \showtwo{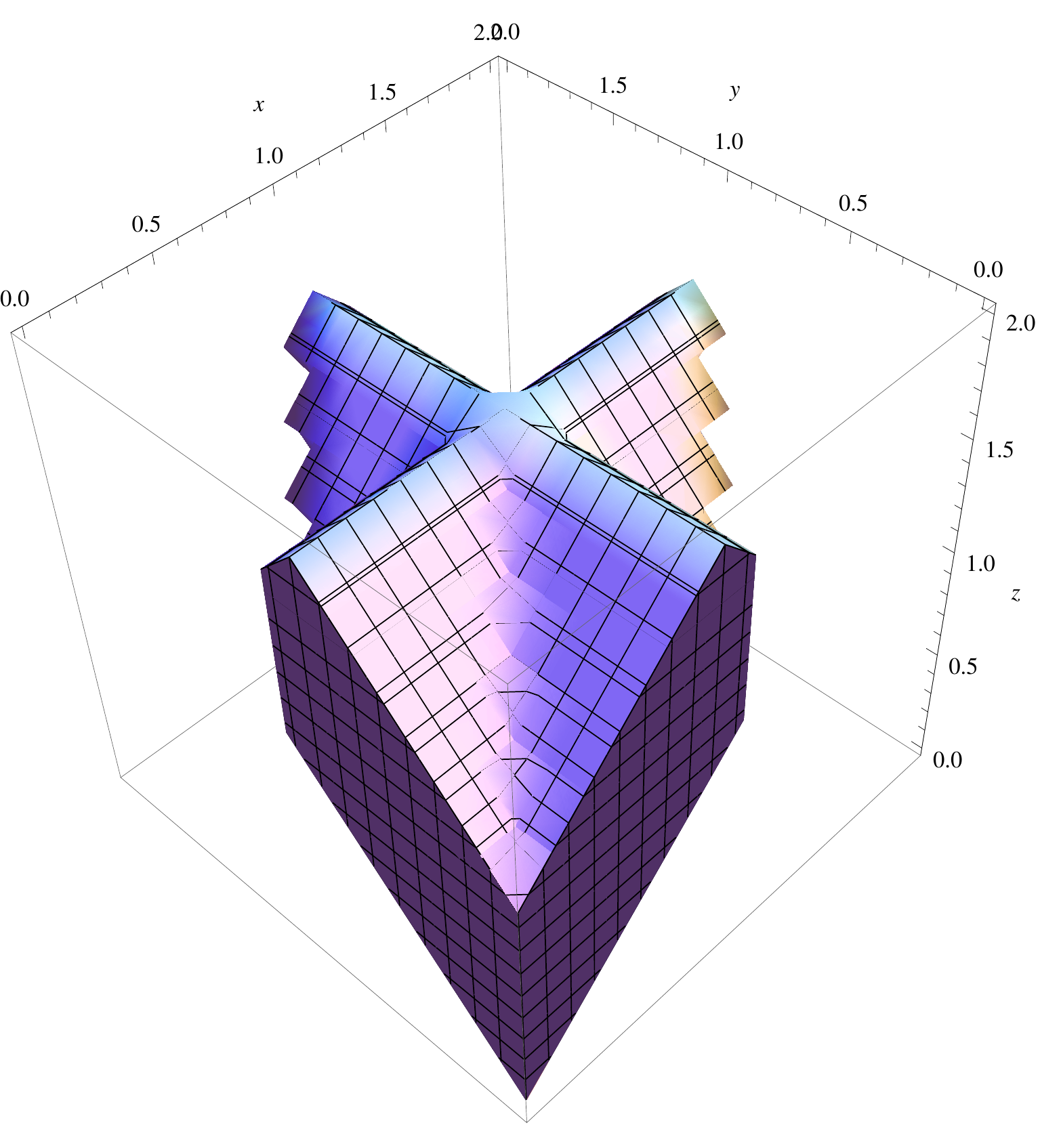}{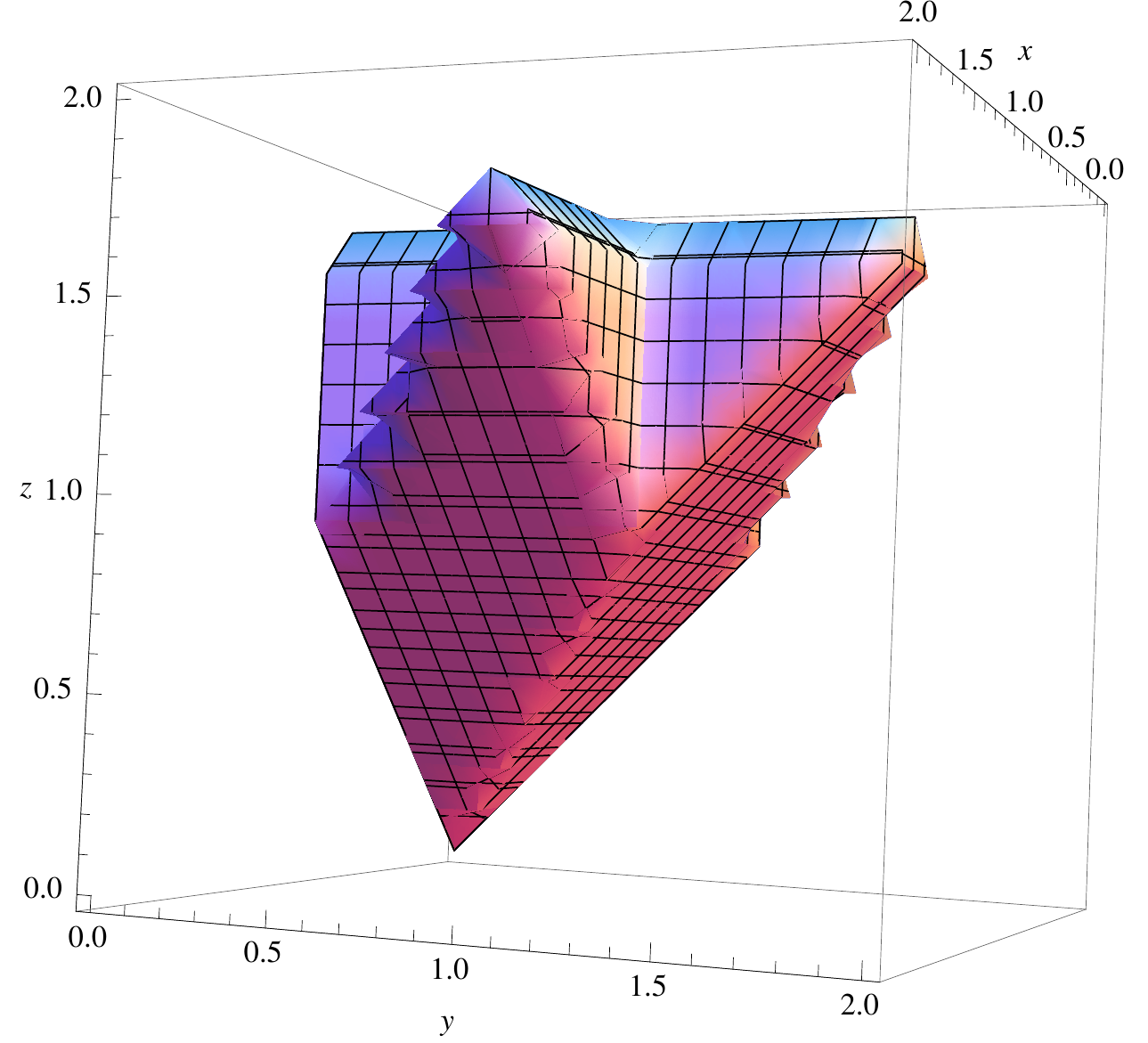}{0.25}
    \caption{{\bf Hilbert Kunz Multiplicity}: view from above and from the back after rotation around the $z$-axis. }
 \end{figure}

%\begin{figure}[tbh]
   % \showone{Fig-side}{0.25}
    %\caption{{\bf Hilbert-Samuel Multiplicity}}
%\end{figure}

%\begin{figure}[tbh]
  % \showone{FSigFigFront}{0.25}
   % \caption{{\bf $F$-Signature}}
 %\end{figure}

% \begin{figure}[tbh]
   % \showone{HKFig-Front}{0.25}
   % \caption{{\bf Hilbert Kunz Multiplicity}}
 %\end{figure}

\pagebreak

\end{document}